\documentclass[11pt,reqno]{amsart}
\usepackage{amssymb,latexsym}
\usepackage{amsmath}
\usepackage{amsthm}

\numberwithin{equation}{section}
\newtheorem{theorem}{Theorem}[section]
\newtheorem{proposition}[theorem]{Proposition}
\newtheorem{lemma}[theorem]{Lemma}

\usepackage{multirow}

\theoremstyle{definition}

\newtheorem{remark}[theorem]{Remark}

\def\R{{\mathfrak R}}

\def\begeq{\begin{equation}}
\def\endeq{\end{equation}}

\def\R{\Bbb R}

\def\dev{\textup{\,d}}
\def\dy{\dev y}

\allowdisplaybreaks

\begin{document}

\title{ New type of solutions for the nonlinear Schr\"odinger-Newton system }
\author{Haixia Chen}
\address{Haixia Chen,
\newline\indent  School of Mathematics and Statistics, Central China Normal University,
\newline\indent Wuhan 430079, P. R. China.
}
\email{hxchen@mails.ccnu.edu.cn}

\author{Pingping Yang$^{\dagger}$}
\address{Pingping Yang,
\newline\indent School of Mathematics and Statistics, Central China Normal University,
\newline\indent Wuhan 430079, P. R. China.
}
\email{ypp15623175603@mails.ccnu.edu.cn}

\begin{abstract}
The nonlinear Schr\"{o}dinger-Newton system
\begin{align*}
 \begin{cases}
 \Delta  u- V(|x|)u + \Psi  u=0,\quad &x\in\R^3,\vspace{0.05mm}\\
\Delta  \Psi+\frac12 u^2=0, &x\in\R^3,
\end{cases}
\end{align*}
is a nonlinear system obtained by coupling the linear Schr\"{o}dinger equation of quantum mechanics with the gravitation law of Newtonian mechanics.~Wei~and~Yan~in (Calc. Var. Partial Differential Equations 37 (2010),423--439) proved that the ~Schr\"{o}dinger~ equation has infinitely many positive solutions in ~$\mathbb{R}^N$~  and these  solutions have polygonal symmetry in the ~$(y_{1}, y_{2})$~ plane and they are  radially symmetric in the other variables.~Duan~ et al. in ( arXiv:2006.16125v1 ) extended the results got by~ Wei~ and~ Yan~ and  they proved that the ~Schr\"{o}dinger~ equation has infinitely many positive solutions in ~$\mathbb{R}^N$~  and these  solutions have polygonal symmetry in the ~$(y_{1}, y_{2})$~ plane and they are  even in ~$y_{2}$~ with one more more parameter in the expression of the solutions.~Hu~ et al. also extended the results got by~Wei~and~Yan~. Under the appropriate assumption on the potential function ~V,~ ~ Hu~et al. in (arXiv: 2106.04288v1) constructed  infinitely many  non-radial positive solutions for  the  Schr\"{o}dinger-Newton  system and these positive solutions have polygonal symmetry in the~$(y_{1}, y_{2})$~plane and they are  even in~$y_{2}$~ and ~$y_{3}$.~
  Assuming that ~$V(r)$~has the following character
\begin{equation*}
    V(r)=V_{1}+\frac{b}{r^q}+O\Big(\frac{1}{r^{q+\sigma}}\Big),
 \quad\mbox{ as  } r\rightarrow\infty,
\end{equation*}
where ~$\frac12\le q<1$~ and ~$b, V_{1}, \sigma$~ are some positive constants, ~$V(y)\ge V_1>0$~, we construct infinitely many  non-radial positive solutions which have polygonal symmetry in the ~$(y_{1}, y_{2})$~ plane and are even in ~$y_{2}$~ for  the  Schr\"{o}dinger-Newton  system by the ~Lyapunov-Schmidt~ reduction method. We extend the results got by ~ Duan~ et al. in ( arXiv:2006.16125v1 ) to  the  nonlinear Schr\"{o}dinger-Newton system. Meanwhile, the solutions constructed by us are   different from what constructed by ~Hu~et al. as there is one more parameter in the expression of the solutions constructed by ~Hu~et al..
\vspace{2mm}

\vspace{2mm}

\vspace{2mm}
\end{abstract}
\thanks{$^{\dagger}$ Corresponding author: Pingping Yang }
\maketitle

\section{introduction}
  The nonlinear Schr\"{o}dinger-Newton system is a nonlinear system obtained by coupling the linear Schr\"{o}dinger equation of quantum mechanics with the gravitation law of Newtonian mechanics, where the wave function $u$  means a stationary  solution for a quantum system describing a  nonlinear modification of the Schr\"{o}dinger equation with a Newtonian gravitational potential representing the interaction of the particle with its own gravitational field. Our aim is to construct  solutions of the  following    nonlinear  Schr\"odinger-Newton equation
    \begin{equation}\label{SN0}
        \begin{cases}
        \Delta  u- V(|x|)u + \Psi  u=0,\quad &x\in\R^3,\\[2mm]
        \Delta\Psi+\frac12u^2=0, &x\in\R^3,
        \end{cases}
    \end{equation}
 where $u\in H^{1}(\R^3)$ and  $V$ is a given external potential and a bounded radially symmetric potential with $V(y)\ge V_1>0$, $\Psi$ is the Newtonian gravitational potential.  The latter model was proposed in \cite{Penrose}.
    We can see  that the second equation of \eqref{SN0} (see \cite{Trudinger1997})  has a unique positive solution $\Psi_{u}\in D^{1,2}(\R^3)$  which has the form as follows
    \begin{equation}\label{psi}
        \Psi_{u}(x)=\frac1{8\pi}\int_{\R^3}\dfrac{u^2(y)}{|x-y|}\dy.
    \end{equation}
So, the system \eqref{SN0} is equivalent to the single nonlocal equation as follows
    \begin{equation}\label{SSN}
        -\Delta u+V(x)u=\frac1{8\pi}
           \Big(
                \int_{\R^3}\frac{u^2(y)}{|x-y|}\dy
            \Big)
            u, \quad x\in\R^3.
    \end{equation}
Obviously, $(u,\Psi_u)$ is a solution of the system \eqref{SN0} if and only if $u$ is a solution of the equation \eqref{SSN}.
   From \cite{lieb,  menzala,Lions0} , we can  make the equation~(\ref{SSN} )~as a special case of the Choquard equation:
\begin{equation*}
        -\Delta u+V(x)u= \Big(\int_{R^N}\frac{A_{\alpha}}{|x-y|^{N-2}}u^{p}(y)dy \Big)|u|^{p-2}u,
           \quad x\in\R^N,
    \end{equation*}
which describes an electron trapped in its own hole, where~$A_{\alpha}=\frac{\Gamma(\frac{N-\alpha}{2})}{\Gamma(\frac{\alpha}{2})\pi^{\frac{N}{2}}2^{\alpha}}$~.  We can also get
\begin{equation*}
        -\Delta u+V(x)u=\frac1{8\pi}
           \Big(
                \int_{\R^3}\frac{u^2(y)}{|x-y|}\dy
            \Big)
            u, \quad x\in\R^3,
    \end{equation*} with ~$N=3$~, ~$p=2$~, ~$\alpha=2.$~
  \par
   The study of standing waves of nonlinear Hartree equations:
 \begin{equation*}
        i\varepsilon\frac{\partial\varphi}{\partial t}=-\varepsilon^{2}\triangle_{x}\varphi+(V(x)+E)\varphi-\frac{1}{8\pi\varepsilon^{2}} \Big(\int_{R^3}\frac{\varphi^{2}(y)}{|x-y|} \Big)\varphi,
           \quad (x,t)\in\R^3\times\R^+,
    \end{equation*}
    can come down to the study of the equation~(\ref{SSN})~, with~$ \varphi(x,t)=e^{-iEt/\varepsilon}u(x)$~, where ~$i$~is the imaginary unit and ~$\varepsilon$~is the Planck constant .
\par
The existence and the uniqueness of ground state solution to the latter problem of the system:
\begin{equation}\label{1}
        \begin{cases}
            \Delta u-u+\Psi u=0,\quad &x\in\R^3,\\
            \Delta\Psi+\frac12 u^2=0, &x\in\R^3,
        \end{cases}
    \end{equation}
have been proven in  \cite{Lions,Lions2} and \cite{Tod99} .
The nondegeneracy of the ground state also has been proven in \cite{WeiJmp}. By \cite{Tod99} and \cite{WeiJmp}, we can know that the following equation:
\begin{equation}\label{U-equation}
  \begin{cases}
   -\Delta  u+  u - \Psi  u=0,\quad x\in\R^3,\\
  \Delta  \Psi+\frac12 u^2=0, \quad x\in\R^3,\\
   u,\ \Psi>0,\quad u(0)=\max_{x\in\R^3}u(x)
\end{cases}
\end{equation}
has a unique radial solution~$(U,\Psi)$~with
\begin{equation*}
   U(x)\rightarrow 0,\ \ \Psi(x)\rightarrow 0, \quad\mbox{ as }|x|\rightarrow\infty.
 \end{equation*}
Furthermore, ~$U$~decreases strictly and
\begin{equation}\label{decay}
  \lim_{|x|\rightarrow \infty}U(x)|x|e^{|x|}=\lambda_{2},\quad \lim_{|x|\rightarrow\infty}\frac{U'(x)}{U(x)}=-1,
\end{equation}
and
\begin{equation}\label{decay1}
    \lim_{|x|\rightarrow \infty}\Psi(x)|x| = \lambda_{3},
 \end{equation}
where~$\lambda_{2},\ \lambda_{3}$~are postive constants.
\par
  Kang  and Wei proved the existence of positive $K-$bump
   solutions to the nonlinear ~Schr\"{o}dinger~ equation:
   \begin{equation}\label{J}
       h^{2}\Delta  u-V(x)u+u^{p} = 0, \quad u>0, \quad x\in\R^N,
    \end{equation}
   concentrating  at local maximum  points of $V$ as $\varepsilon\to0$ in  \cite{Kang}. Furthermore, each pair of bumps has a  strong interaction. ~Wei ~and ~Winter~extended the results to the  following singularly perturbed  Schr\"{o}dinger-Newton  problem
    \begin{equation}\label{PSN}
        -\varepsilon^2\Delta  u+V(x)u = \frac{1}{8\pi\varepsilon^2}
            \Big(
                \int_{\R^3}\frac{u^2(y)}{|x-y|}\dy
           \Big)
            u, \quad x\in\R^3,
    \end{equation}
    where $\varepsilon>0$ is a  parameter and $\inf_{\R^3}V>0$.  Also, they proved the existence of positive $K-$bump solutions to \eqref{PSN} concentrating  at local maximum (minimum) or nondegenerate critical points of $V$ as $\varepsilon\to0$. Furthermore, it is  indicated that there is a  strong interacting between each pair of bumps.    ~Luo~  and ~Peng~et al. have shown the solutions to  \eqref{PSN} is unique with ~$\varepsilon$~ small enough  by using local Pohozaev identities, blow-up analysis and the maximum principle in \cite{Luo}. Recently, ~Guo~  and ~Luo~et al. proved the existence and local uniqueness of the normalized single peak solutions for a Schr\"{o}dinger-Newton system concentrating  at degenerate critical points of the potential $V$ and the non-existence of the normalized  multi-peak solutions  concentrating  at degenerate critical points of the potential $V$  for a Schr\"{o}dinger-Newton system by finite dimensional reduction method and local Pohozaev identities  in \cite{Guo1}.
    \par
    The paper with a more complex concentration structure is to construct a new family of entire solutions of \eqref{SN0} with a radial potential $V(r)$, that is,
    \begin{equation}\label{SN}
        \begin{cases}
            \Delta u-V(|x|)u+\Psi u=0,\quad &x\in\R^3,\\
            \Delta\Psi+\frac12 u^2=0, &x\in\R^3,
        \end{cases}
    \end{equation}
    or equivalently,
    \begin{equation}\label{SN1}
        -\Delta  u+V(|x|)u = \frac1{8\pi}
       \Big(
            \int_{\R^3}\frac{u^2(y)}{|x-y|}\dy
        \Big)
        u, \quad x\in\R^3.
    \end{equation}
    For the purpose, we assume that $V$ has the following behavior at infinity:
    \begin{itemize}
      \item[($\rm H_{a}$)] There exist some constants $b$, $\sigma$, $V_1>0$  and $\frac12\le q<1$, such that $V(x)\ge V_1$ and
          \begin{equation}
            V(r)=V_{1}+\frac{b}{r^q}+O\Big(\frac1{r^{q+\sigma}}\Big), \quad\mbox{ as } r\to\infty.
          \end{equation}
    \end{itemize}
    \par
    Observe that a direct scaling argument can enable us to just consider $V_{1}=1$. In the same assumption, Hu and Jevnikar et al. in \cite{HuY}  produce infinitely many non-radial solutions to \eqref{SN0} with high energy. From \cite{Tod99} and \cite{WeiJmp}, we can know the equation ~(\ref{U-equation})~has the unique radial solution ~$(U,\Psi),$~ so the equation \eqref{SN1} has the unique radial solution with ~$V(x)\equiv1$~. For any large integer $m$ they constructed a solution $u_m$ looking like   a sum of  $m$   {\it bumps}  $U(x-x_i^\star)$,
    \begin{equation}\label{uuk}
        u^\star_{m} (y) \sim \sum_{i=1}^m U(y-x_{i}^\star ),
    \end{equation}
    where the location points $x_{i}^\star$ are distributed along the vertices of a regular $m$-polygon
        $$x_{i}^\star=r\Big(\cos\frac{2(i-1)\pi}m, \sin\frac{2(i-1)\pi}m, 0\Big), \quad\text{ for }i=1,\ldots, m$$
    with large radius $r\sim(m\ln m)^{\frac{1}{1-q}}$ as $m\to\infty$. These solutions  have polygonal symmetry in the~$(y_{1}, y_{2})$~plane, which are  even in~$y_{2}$~ and ~$y_{3}$~and  have ~$m$~ bumps. Very recently, in \cite{Duan}, Duan and Musso used a new way to construct a new family of solutions of
    \begin{align*}
        \begin{cases}
            -\Delta u+V(y)u=u^p, \quad u>0, \text{ in }\R^N,\\
            u\in H^1(\R^N).
        \end{cases}
    \end{align*}
These solutions  have polygonal symmetry in the~$(y_{1}, y_{2})$~plane, which are  even in~$y_{2}$~ and have ~$2m$~ bumps.

  Recently, Gao and Yang in \cite{Gao} proved that the nonlinear Choquard equation:
 \begin{align}\label{sn9}
    \begin{cases}
             -\Delta u+V(|x|)u =
       \big(
            \int_{\R^3}\frac{u^2(y)}{|x-y|}\dy
        \big)u, \quad x \in \R^3,\\
           u\in H^{1}(\R^3),
    \end{cases}
\end{align}
has infinitely many non-radial positive solutions, where the potential $V$ satisfies:
 \begin{equation}\label{Ha3}
            V(r)=V_{1}+\frac{b}{r^q}+O\Big(\frac1{r^{q+\sigma}}\Big), \quad\mbox{ as } \quad r\to\infty,
\end{equation}
where~$b,\ \sigma,\ V_{1}>0$~and~$q\geq3$~.
\par
    In the paper, the main result is that we extend the result got by ~ Duan~ et al. in \cite{Duan} to  the  nonlinear Schr\"{o}dinger-Newton system. As in \cite{Duan}, let $m$ be an integer and define the points ~$\overline{x}_{i}=\overline{x}_{i}(r,t,m)$~, ~$\underline{x}_{i}=\underline{x}_{i}(r,t,m)$~:
    \begin{equation} \label{overunder}
        \begin{cases}
            \overline{x}_{i}=r\Big(\sqrt{1-t^2}\cos\frac{2(i-1)\pi}m, \sqrt{1-t^2} \sin\frac{2(i-1)\pi}m,t\Big), \quad i=1,\ldots, m, \\[5mm]
            \underline{x}_{i}=r\Big(\sqrt{1-t^2}\cos\frac{2(i-1)\pi}m,\sqrt{1-t^2}  \sin\frac{2(i-1)\pi}m, -t\Big), \quad i=1,\ldots,m.
        \end{cases}
    \end{equation}
    The parameters $t$ and $r$ are positive numbers and the range of them are chosen as follows
     \begin{equation}\label{par1}
        t\in\big[\alpha_{1}(\ln m)^{-\frac12},\alpha_{2}(\ln m)^{-\frac12}\big], \quad  r\in[\beta_{1}(m\ln m)^{\frac1{1-q}}, \beta_{2}(m\ln m)^{\frac1{1-q}}],
     \end{equation}
    for $\alpha_{1},\alpha_{2},\beta_{1},\beta_{2}$ fixed positive constants, independent of $m$.
    \par
    Define the approximate solution  as
    \begin{align}\label{Wrt1}
        W_{r,t}(y)=\sum_{i=1}^mU_{\overline{x}_{i}}(y)+
        \sum_{i=1}^m U_{\underline{x}_{i}}(y),
    \end{align}
    where $U_{\overline{x}_{i}}(y)=U(y-\overline{x}_{i}), U_{\underline{x}_{i}}(y)=U(y-\underline{x}_{i})$ and $m$ is large enough.

    In the paper, we will prove that for any $m$ sufficiently large problem \eqref{SN0} has a new type of solutions $u_{m} $ with the form
    \begin{equation}\label{um}
        u_{m}(y)\sim W_{r,t}(y),
    \end{equation}
    as $m\to\infty$. The solutions will have  polygonal symmetry in the $(y_{1},y_{2})$-plane and be even in the $y_{3}$ direction. Particularly, they do not belong to the same class of symmetry as the solutions built in \cite{HuY}, because we have one more parameter. In fact, if we take $t=0$ in \eqref{overunder}, we have $\overline{x}_{i}=\underline{x}_{i}$ for any $i$ and the two constructions  \eqref{uuk} and \eqref{Wrt1}-\eqref{um} are the same.  Assuming that the range of the parameters $t$ and $r$ is the same as \eqref{par1}, the two constructions \eqref{uuk} and \eqref{um}-\eqref{Wrt1} are different. The main distinction between the present construction and the one in \cite{HuY} is that there exist two parameters $r,t$ to choose in the locations $\underline{x}_{i},    \overline{x}_{i}$ of the bumps in \eqref{Wrt1}. Although, we use the method in \cite{Duan}, the non-local term brings some new difficulties, which involves many complex and technical estimates. We would like to stress that the energy estimate is different from \cite{Duan} because of the appearance of the non-local term. And we have to choose different range of $r$ and $t$ to ensure the mapping is contractible in the proof of theorem $\ref{main1}$. Compared with \cite{Duan}, we also check that the mapping is contractible by a different method.
\par
  In the following subsection, we will discuss  our result in details.
\par
    In this paper, we use $C,C_{i}$, or $\theta,\theta_{i},\mu,\mu_{i}$, $i=0,1,\ldots$ to denote fixed constants. Moreover, we also use the common notation by writing $O_{m}(F(r,t)),o_{m}(F(r,t))$ for the functions satisfying
    \begin{align*}
        \text{if} \quad G(r,t) \in O_{m}(F(r,t)) \quad \text{then }\quad  {\lim_{m \to +\infty}} \Big|\,  \frac{G(r,h)}{F(r,t)} \, \Big|  \leq C< + \infty,
    \end{align*}
    and
    \begin{align*}
        \text{if} \quad G(r,t) \in o_{m}(F(r,t))\quad \text{then }\quad  {\lim_{m \to +\infty}}\frac{G(r,t)}{F(r,t)} =0.
    \end{align*}

\subsection{Main result and scheme of the proof }
\textsc{}
\vspace{2mm}

    For $ i= 1,\cdots, m$, we divide    $\mathbb{R}^3$ into $m$ parts:
\begin{align*}
\Omega_{i} : = &\Big\{  y = (y_{1}, y_{2}, y_{3}) \in \mathbb{R}^3 \nonumber
:       \langle  \frac{ (y_{1}, y_{2})}{|(y_{1}, y_{2}) |},           ( \cos{ \frac{2(i-1) \pi }{m}     },    \sin{ \frac{2( i-1) \pi }{m}     } )   \rangle_{\mathbb{R}^2 }\geq \cos{ \frac \pi m}          \Big\},
\end{align*}
where $ \langle   ,  \rangle_{\mathbb{R}^2 } $  denotes  the dot product in $\mathbb{R}^2$.
 For $\Omega_{i}$,   we divide  it into  two parts:
\begin{align*}
\Omega_{i}^+ = & \Big\{  y:    y = (y_{1}, y_{2}, y_{3})  \in   \Omega_{i}, y_{3}\geq0 \Big\},
\end{align*}
\begin{align*}
\Omega_{i}^- = & \Big\{  y:  y = (y_{1}, y_{2}, y_{3})   \in   \Omega_{i}, y_{3}<0 \Big\}.
\end{align*}
We have that
\begin{align*}
\mathbb{R}^3 =  \cup_{i=1}^m   \Omega_{i},   \quad  \Omega_{i} =     \Omega_{i}^+    \cup    \Omega_{i}^-
\end{align*}
and  the  interior of
\begin{align*}
 \Omega_{i}  \cap   \Omega_{j},        \quad   \Omega_{i}^+    \cap    \Omega_{i}^-
\end{align*}
   are   empty sets for $j\neq i$.

\medskip

Now we  define the  symmetric Sobolev space:
\begin{align*}
H= \Big\{  & u : u \in H^1(\mathbb{R}^3),    \text{ $u$   is even in   $ y_{2}, $}  \quad   u \Big( \sqrt{y_{1}^2+y_{2}^2} \cos \theta, \sqrt{y_{1}^2+y_{2}^2} \sin \theta,   y_{3}  \Big)
\nonumber
\\
     &=  u \Big( \sqrt{y_{1}^2+y_{2}^2} \cos { \big( \theta+   \frac{2i\pi}{m}  \big)  }, \sqrt{y_{1}^2+y_{2}^2} \sin { \big( \theta+   \frac{2i\pi}{m}  \big)  }, y_{3} \Big)         \Big\},
\end{align*}
where $ \theta =  \arctan{\frac {y_{2}}{ y_{1}} }$.
\par
\vspace{2mm}
   In the paper, we always assume
\begin{align}\label{H2}
(r,t) \in \mathbb{S}_{m}
 & =: \Big [  \Big(\Big(\frac{A_{1}}{16q( \pi)^{2}    }\Big)^{\frac{1}{1-q}}- \beta_{0}  \Big)   (m \ln m)^{\frac{1}{1-q}},   \,       \Big(\Big(\frac{A_{1}}{16q( \pi)^{2}    }\Big)^{\frac{1}{1-q}}+ \beta_{0}  \Big)  ( m \ln m)^{\frac{1}{1-q}} \Big]  \nonumber
 \\[1mm]
 & \qquad
 \times\Big[   \Big(1 - \alpha_{0} \Big) (  \ln m)^{-\frac{1}{2}},    \,   \Big(1 + \alpha_{0} \Big) (  \ln m)^{-\frac{1}{2}} \Big],
\end{align}
for some $ \alpha_{0}, \beta_{0}>0$  small enough, and independent of $m$.  We can refer to Remark \ref{remark6} to discuss the  assumption \eqref{H2}  for $(r,t)$.

The following is our main result.
\vspace{3mm}
\begin{theorem}\label{main1}
Assume that  $ V(|y|) $ satisfies  $(\rm H_{a})$ and the parameters $(r,t)$ satisfy \eqref{H2}. So there exists an integer $m_{0} $,   such that for any integer $m\geq m_{0}$,
\eqref{SN0} has a solution $u_{m} $ of the form
\begin{align}\label{u_k}
u_{m} = W_{r_{m}, t_{m}}(y) + \omega_{m}(y),
\end{align}
where $ \omega_{m} \in H$,  $(r_{m},s_{m}) \in \mathbb{S}_{m}  $ and  $ \omega_{m} $ satisfies
\begin{align*}
  \int_{\mathbb{R}^3 }  \big( | \nabla {\omega}_{m} |^2 + V(y) |\omega_{m}|^2 \big)  \to 0, \quad {\mbox {as}} \quad m \to \infty.
\end{align*}
\qed
\end{theorem}
We want to prove  Theorem  $ \ref{main1} $ by using the Lyapunov- Schmidt reduction technique adapted to our context as developed in \cite{wei&yan}. We  sketch the scheme of the proof briefly. The critical point $u$ of the energy functional $I$ is defined as follows:
\begin{equation}\label{functional}
 I(u)=\frac{1}{2}\int_{\R^3}(|\nabla u|^2 +V(|x|)u^2)-\frac{1}{32\pi}\int_{\R^3}\int_{\R^3}\frac{u^2(y)u^2(x)}{|x-y|},
\end{equation}
which corresponds to a solution for \eqref{main1}. For any m sufficiently large and  for any ~$(r,t) \in \mathbb{S}_{m}$~, assuming that the critical point of the energy functional $I$  has the following form ~$u= W_{r,t} + \phi$~, the solution $u$ will have the form ~$u= W_{r,t}   + \phi$~.
\par
\noindent
Furthermore, assume that  $\phi\in H^1(\R^3)$ satisfies the following eigenvalue problem :
\begin{equation*}
   -\Delta  \phi+\phi=\frac{1}{8\pi}\Big(\int_{\R^3}\frac{U^2(y)}{|x-y|}dy \Big) \phi +\frac{1}{4\pi}\Big(\int_{\R^3}\frac{U(y)\phi(y)}{|x-y|}dy\Big) U.
\end{equation*}
Then, $ \phi $ can be spanned by  $ \{ \frac{\partial U}{\partial {x_{1}}  }, \frac{\partial U}{\partial {x_{2}}  }, \frac{\partial U}{\partial {x_{3}} } \}$, that is to say
\begin{equation*}
  \phi\in span\left\{\frac{\partial U}{\partial x_{i}}, i=1,2,3\right\}.
\end{equation*}

We define the inner product and norm on the space $H^1(\R^3)$ by
  \begin{equation*}
   \langle u,v\rangle=\int_{\R^3}(\nabla u \nabla v+V(|x|)u v)\quad \mbox{and}\quad
     \|u\| =\langle u,u \rangle ^{\frac{1}{2}},
  \end{equation*}
   respectively. Noting that $V$ is bounded, so the norm $\|\cdot\|$
is equivalent to the standard norm.
  Particularly, by Lemma \ref{lemma0}, we have
\begin{equation}\label{psi-norm}
   ||\Psi_u||_{D^{1,2}(\R^3)}^2=\int_{\R^3}\Psi_u u^2 \leq C\|u\|^4_{L^{\frac{12}5}(\R^3)}\leq C||u||^4.
 \end{equation}
 Throughout the paper,  $C$  denotes various positive constants whose exact value is not essential.
We define
\begin{align*}
R(\phi ) =     I (W_{r,t}   + \phi ),       \quad  \phi \in    \mathbb{E}_{1}.
\end{align*}
The space $\mathbb{E}_{1}$ is defined later. \medskip
 For $ i = 1, \cdots, m, $    we define
 \begin{align*}
 \overline{ \mathbb{Z}}_{1i}  =    \frac{\partial U_ {\overline{x}_{i}} }{\partial r}, \quad \underline{ \mathbb{Z}}_{1i}  =    \frac{\partial U_ {\underline{x}_i} }{\partial r}, \quad  \overline{ \mathbb{Z}}_{2i}  =    \frac{\partial U_ {\overline{x}_i} }{\partial t},  \quad \underline{ \mathbb{Z}}_{2i}  =    \frac{\partial U_ {\underline{x}_{i}} }{\partial t}.
 \end{align*}
 The constrained  space is defined as follows
\begin{align}\label{SpaceE}
\mathbb{E}_{1}  = \Big\{ v:   v\in H_{a},      \quad \int_{{\mathbb{R}}^3 } &T[ U_{\overline{x}_{i}}^{2} ]  \overline{ \mathbb{Z}}_{\ell i} v +2T[ U_{\overline{x}_{i}} \overline{ \mathbb{Z}}_{\ell i}] U_{\overline{x}_{i}}  v= 0  \quad\text{and}    \nonumber
\\[2mm]
  \quad  & \int_{{\mathbb{R}}^3 } T[ U_{\underline{x}_{i}}^{2} ]\underline {\mathbb{Z}}_{\ell i}  v+2T[ U_{\underline{x}_{i}} \underline {\mathbb{Z}}_{\ell i}]U_{\underline{x}_{i}}   v = 0, \quad i = 1, \cdots,m, \quad \ell = 1,2      \Big\},
 \end{align}
 where
 \begin{equation}\label{T[]}
   T[u v](x)=\frac{1}{4\pi} \int_{\R^3}\frac{u(y)v(y)}{|x-y|}dy.
 \end{equation}
Expand $ R(\phi )$ in the following:
\begin{align*}
R(\phi ) = R(0) + {  \bf{ l } }  (\phi) + \frac 12  \langle{\bf{L}} \phi,  \phi\rangle  -  {  \bf{ N } }(\phi),    \quad \phi \in   \mathbb{E}_{1},
\end{align*}
where

\begin{align*}
{ \bf{ l } }  (\phi) & = \sum_{i =1}^m   \int_{  \mathbb{R}^3} \big( V(|y|) - 1      \big)  \big( U_{ \overline{x}_{i}  } + U_{\underline{x}_{i}  }\big)  \phi\\
 &\quad+ \frac{1}{2}  \int_{  \mathbb{R}^3} \Big(  \sum_{i =1}^m T[(U_{ \overline{x}_{i} } )^2]  U_{ \overline{x}_{i}  } + \sum_{i =1}^m T[(U_{\underline{x}_i  }   \big)^2]U_{\underline{x}_{i} }-T[W_{r,t}^2 ]  W_{r,t}   \Big) \phi,
\end{align*}

and
\begin{align}\label{N0}
{ \bf{ N } }(\phi) &=  \frac{1}{32\pi}\Big(\int_{\R^3}\int_{\R^3}\frac{(                       W_{r,t}+\phi)^2 (y)(          W_{r,t}+\phi)^2(x)}{|x-y|}-\int_{\R^3}\int_{\R^3}\frac{( W_{r,t})^2 (y)( W_{r,t})^2(x)}{|x-y|}\Big)\nonumber
   \\[0.02mm]&
\quad
-\frac{1}{2}  \int_{  \mathbb{R}^3} T[W_{r,t}^2 ]W_{r,t}\phi-\frac{1}{16\pi}\int_{\R^3}\Big(\int_{\R^3}\frac{( W_{r,t})^2 (y)}{|x-y|}dy\Big)\phi^2\nonumber
   \\[0.02mm]&
\quad-\frac{1}{8\pi}\int_{\R^3}\Big(\int_{\R^3}\frac{ W_{r,t}\phi}{|x-y|}dy\Big)W_{r,t}\phi.
\end{align}
Furthermore, ${\bf{L}} $ is  a   linear operator  from $  \mathbb{E}_{1} $ to $ \mathbb{E}_{1}$,  which satisfies  $   \text{for all }   ~ v_{0}, v_{1} \in       \mathbb{E}_{1} ,$
\begin{align*}
 \langle{\bf{L}} v_{0},  v_{1} \rangle & =     \int_{{\mathbb{R}}^3 }      \nabla {v_{0}} \nabla {v_{1}}  + V (|y|) v_{0} v_{1}  -\frac{1}{8\pi}\int_{\R^3}\Big(\int_{\R^3}\frac{( W_{r,t})^2 (y)}{|x-y|}dy\Big) v_{0} v_{1}
 \\&\quad-\frac{1}{4\pi}\int_{\R^3}\Big(\int_{\R^3}\frac{ W_{r,t}v_{0}}{|x-y|}dy\Big)W_{r,t}v_{1}.
 \end{align*}
Since  $W_{r,t}$ is bounded and has the symmetries of the space $H_{a}$, it is easy to  prove that $  {\bf L}  $  is a  bounded   linear operator from $  \mathbb{E}_{1}$ to $\mathbb{E}_{1}$ by  Hardy-Littlewood-Sobolev inequality.
We are going to prove that $ {  \bf{ l } }  (\phi)$ is a bounded linear functional in $ \mathbb{E}_{1}$. Thus  there exists an ${  \bf{ l } }_m\in  \mathbb{E}_{1}$, such that  $$   {  \bf{ l } }  (\phi)=    \langle  {  \bf{ l } }_{m}, \phi  \rangle.$$
Then  a critical point of $ R(\phi )$  is also  a solution to
\begin{align}\label{equivalentequat}
{  \bf{ l } }_m +{\bf{L}} \phi-  {  \bf{ N } }' (\phi) = 0.
\end{align}
Thus the function $u$ which has the form $u= W_{r,t}   + \phi$ will be a solution to \eqref{SN0} if   $ \phi$ is a solution to \eqref{equivalentequat}.
Firstly the strategy now consists  presenting that, for any $ (r,t) \in \mathbb{S}_m $, there is a function  $\phi_{r,t} $  solution  to \eqref{equivalentequat}  in the  space $\mathbb{E}_{1}$ (see Proposition \ref{propos2}).   Secondly, we will
 reduce the problem to find a critical point $u$ of  $I(u)$ to  the problem of finding a stable critical point $(r^\star,t^\star)$ of the  function $$  \overline{F}(r,t)  =  I (W_{r,t}   + \phi_{r,t}  ).$$
We will present that there is such a critical point  in the set $\mathbb{S}_m$. Lots of properties of the solutions  follows by the construction
(see \eqref{H2}).

 \vspace{3mm}

\noindent{\textbf{Plan of the paper}}\\
We write the paper in the following. In section \ref{s2},
we do the finite-dimensional reduction. And we are going to prove  Theorem \ref{main1} in section \ref{s3}.
In the Appendix \ref{sa}, we give some known results and some technical estimates.
Finally, in the Appendix \ref{sb} we make some energy expansions.

\section{  Finite-dimensional reduction}\label{s2}
\medskip

The following lemma proves the existence and boundedness of inverse operator of ${\bf{L}}$ in    $ \mathbb{E}_{1}. $

\medskip
\begin{lemma}\label{lemmainverse}
There is a constant $ \xi > 0$, independent of $m$, such that for any $ (r,t)\in  \mathbb{S}_m$
$$  \|{\bf{L}} v\| \geq \xi \| v\|, \quad v\in \mathbb{E}_{1}. $$
\end{lemma}

\begin{proof}
We prove it by contradiction.  Assuming that
there are  $t_{m}, r_{m}  \in   \mathbb{S}_{m},       v_{m} \in  \mathbb{E}_{1}  $ satisfying
\begin{align*}
\|{\bf{L}} v_{m}\| =  o_{m}(1) \| v_m\|,\quad m\rightarrow +\infty .
\end{align*}
Then we can easily get
\begin{align*}
\langle{\bf{L}} v_{m},      \varphi  \rangle     =     o_{m}(1) \| v_{m}\| \,  \|    \varphi    \|,    \quad \forall\,  \varphi \in  \mathbb{E}_{1}.
\end{align*}
Similar to \cite{wei&yan}, we assume $\| v_{m}\|^2   =m$.

Firstly, we claim that  if $\psi\in \mathbb{E}_{1}$, then $T[U\psi]$ satisfies
    \begin{equation}\label{pro-1}
T[U\psi] \mbox{ is even in   } y_{2},
    \end{equation}
  and
\begin{align}\label{pro-2}
& T[U\psi]\Big(\sqrt{y_{1}^{2}+y_{2}^{2}}\cos\theta,\sqrt{y_{1}^{2}+y_{2}^{2}}\sin\theta,y_{3}\Big)\nonumber
 \\[1mm]
& =T[U\psi]\Big(\sqrt{y_{1}^{2}+y_{2}^{2}}\cos\Big(\theta+\frac{2\pi i}{m}\Big),\sqrt{y_{1}^{2}+y_{2}^{2}}\sin\Big(\theta+\frac{2\pi i}{m}\Big),y_{3}\Big).
\end{align}
 Indeed, define $h=(h_{ij})\in SO(3)$ with $h=$diag$(1,-1,1).$ Obviously, $U\in\mathbb{E}_{1}$, which implies  $U(x)=U(hx)$, so is $\psi$. Then, we have
 \begin{align*}
  T[U\psi](h x)&=\frac{1}{4\pi}\int_{\R^3}\frac{U(y)\psi(y)}{|hx-y|}dy=-\frac{1}{4\pi}\int_{\R^3}\frac{U(h^{-1}y)\psi(h^{-1}y)}{|x-h^{-1}y|}dy
 \\& =\frac{1}{4\pi}\int_{\R^3}\frac{U(y)\psi(y)}{|x-y|}dy=T[U\psi](x),
 \end{align*}
  that is equivalent to  \eqref{pro-1}.  Similarly,  define  with the form as follows:
  \begin{equation*}
  h=
 \left(
  \begin{array}{ccc}
  \cos(\frac{2\pi i}{m})&-\sin(\frac{2\pi i}{m})& 0\\[1mm]
  \sin(\frac{2\pi i}{m})&cos(\frac{2\pi i}{m})& 0\\[1mm]
  0&0&1\\
  \end{array}
  \right).
\end{equation*}
So we can get \eqref{pro-2}.
With the symmetric property, we can get
\begin{align} \label{operatorL}
\langle{\bf{L}} v_{m},      \varphi  \rangle   &=   \int_{{\mathbb{R}}^3 }   \nabla {v_{m} } \nabla {\varphi }  + V (|y|) v_{m}  \varphi   -\frac{1}{8\pi}\int_{\R^3}\Big(\int_{\R^3}\frac{( W_{r,t})^2 (y)}{|x-y|}dy\Big) v_{m}  \varphi\nonumber
 \\[0.02mm]&\quad -\frac{1}{4\pi}\int_{\R^3}\Big(\int_{\R^3}\frac{ W_{r,t}v_{m}   }{|x-y|}dy\Big)W_{r,t}\varphi   \nonumber
 \\[0.02mm]
& = m     \int_{ \Omega_{1}}  \nabla {v_{m} } \nabla {\varphi }  + V (|y|) v_{m}  \varphi  -     \frac{1}{8\pi}\int_{ \Omega_1}\Big(\int_{\R^3}\frac{( W_{r,t})^2 (y)}{|x-y|}dy\Big)  v_{m}  \varphi
\nonumber
 \\[0.02mm]&
 \quad-\frac{1}{4\pi}\int_{ \Omega_{1}}\left(\int_{\R^3}\frac{ W_{r,t} v_{m}  }{|x-y|}dy\right)W_{r,t}\varphi    \nonumber
\\[0.02mm]
& =o_{m}(1) \| v_{m}\| \,  \|    \varphi    \|   =o(\sqrt m)  \,  \|    \varphi    \| .
\end{align}
Moreover we have
\begin{align}\label{norm}
  \int_{ \Omega_{1}}  | \nabla {v_{m} } |^2   + V (|y|) v_{m}^2    = 1.
\end{align}
Inserting  $  \varphi  =  v_{m} $ into \eqref{operatorL},    we can get immediately
\begin{align*}
  \int_{ \Omega_{1}}   |  \nabla {v_{m} } |^2 + V (|y|) v_{m}^2 -      \frac{1}{8\pi}\int_{ \Omega_{1}}\Big(\int_{\R^3}\frac{( W_{r,t})^2 (y)}{|x-y|}dy\Big)  v_{m} ^2 -\frac{1}{4\pi}\int_{ \Omega_{1}}\Big(\int_{\R^3}\frac{ W_{r,t} v_{m}  }{|x-y|}dy\Big)W_{r,t}v_{m} = o_{m}(1).
\end{align*}
Denote
$$ \overline{v}_{m} (y) = v_{m} (y+  \overline{x}_{1} ).  $$
For this sequence $ \overline{v}_{m} (y) $, we are able to  prove that  $ \overline{v}_{m} (y) $ has boundedness in $  H^1_{loc}(\mathbb{R}^3). $
Actually,
for  any $ \overline{R}>0$, since $ | \overline{x}_{2} -  \overline{x}_{1}| = 2r  \sqrt{1-t^2} \, \sin{      \frac \pi m } \geq  \frac{n}{4} \ln m,  $
we can choose $m$ sufficiently large such that  $ B_{\overline{R}(   \overline{x}_{1}) } \subset \Omega_{1}$.      Then, we can get
\begin{align}
  \int_{  B_{\overline{R} }(   0 ) }  \big ( \,     | \nabla { \overline{v}_{m} } |^2   + V (|y  |)  \overline{v}_{m}^2 \big ) & =  \int_{  B_{\overline{R}(   \overline{x}_{1}) }}  \big (  \,    | \nabla {  v_{m} } |^2   + V (|y- \overline{x}_{1}  |) { v}_{m}^2\, \big )  \nonumber
  \\[2mm]
    & \leq     \int_{ \Omega_{1}}  \big (\,      | \nabla {{ v}_{m} } |^2   + V (|y- \overline{x}_{1}|)  { v}_{m}^2 \,  \big )\leq 1.
\end{align}
 So we have
 \begin{align}\label{condition0}
 \overline{v}_m \rightarrow   \bar{ v }\quad  \text{ weakly in } H^1_{loc}(\mathbb{R}^3),
 \end{align}
 and  \begin{align}\label{condition5}
 \overline{v}_{m} \rightarrow  \overline{v} \quad  \text{ strongly in } L^2_{loc}(\mathbb{R}^3).
 \end{align}
Since $\overline{v}_{m}$ is even in $y_{2}$,  we know that  $ {\bar v}$ is even in  $y_{2}$.
 By the orthogonal conditions of functions in $\mathbb{E}_{1} $
\begin{align*}
\int_{{\mathbb{R}}^3 } &T[ U_{\overline{x}_1}^{2} ] \frac{\partial  U_ {\overline{x}_{1}} }     {  \partial r }  v_{m} +2T[ U_{\overline{x}_{1}} \frac{\partial  U_ {\overline{x}_{1}} }     {  \partial r } ] U_{\overline{x}_{1}} v_{m}= 0,
\end{align*}
we know the following identity
\begin{align*}
\frac{\partial  U_ {\overline{x}_{1}} }     {  \partial r } =  - \sqrt{1-t^2} \,     \frac{ \partial U_ {\overline{x}_{1}} }{\partial {y_{1}} }     - t \,   \frac{ \partial U_ {\overline{x}_{1}} }{\partial {y_{3}} },
\end{align*}
 so we can obtain
\begin{equation} \label{condition1}
\int_{{\mathbb{R}}^3 } T[ U^2 ](\sqrt{1-t^2}  \frac{ \partial U }{\partial {x_{1}} } +t  \frac{ \partial U }{\partial {x_{3}} })\overline{ v }_{m} dx +2\int_{{\mathbb{R}}^3 } T[ U(\sqrt{1-t^2}  \frac{ \partial U }{\partial {y_{1}} } +t  \frac{ \partial U }{\partial {y_{3}} })]  U \overline{ v }_{m} dx = 0.
\end{equation}
Similarly, combining
\begin{align*}
\int_{{\mathbb{R}}^3 } &T[ U_{\overline{x}_{1}}^{2} ] \frac{\partial  U_ {\overline{x}_{1}} }     {  \partial t }  v_{m} +2T[ U_{\overline{x}_{1}} \frac{\partial  U_ {\overline{x}_{1}} }     {  \partial t } ] U_{\overline{x}_{1}} v_{m}= 0,
\end{align*}
and
\begin{align*}
\frac{\partial U_ {\overline{x}_{1}} }{\partial t } =    \frac{r t}{ \sqrt{1-t^2} }  \frac{ \partial U_ {\overline{x}_{1}} }{\partial {y_{1}} }- r \frac{ \partial U_ {\overline{x}_{1}} }{\partial {y_{3}} },
\end{align*}
we can obtain
\begin{equation} \label{condition2}
 \int_{{\mathbb{R}}^3 } T[ U^2 ](\frac{ r t}{ \sqrt{1-t^2} }  \frac{ \partial U }{\partial {x_{1}} }- r \frac{ \partial U }{\partial {x_{3}} } )\overline{ v }_{m} dx +2\int_{{\mathbb{R}}^3 } T[ U(\frac{ r t}{ \sqrt{1-t^2} }  \frac{ \partial U }{\partial {y_{1}} }- r \frac{ \partial U }{\partial {y_{3}} }]  U \overline{ v }_{m} dx = 0.
 \end{equation}
From \eqref{condition1}, \eqref{condition2},  we can get
\begin{align*}
 \int_{{\mathbb{R}}^3 } T[ U^2 ]  \frac{ \partial U }{\partial {x_{1}} }\overline{ v }_{m} dx +2\int_{{\mathbb{R}}^3 } T[ U \frac{ \partial U }{\partial {y_{1}} }]  U \overline{ v }_{m} dx = 0.
\end{align*}
 Letting  $m \rightarrow   + \infty$,  we can get
 \begin{align}\label{condition3}
 \int_{{\mathbb{R}}^3 } T[ U^2 ]  \frac{ \partial U }{\partial {x_{1}} }\overline{ v }dx +2\int_{{\mathbb{R}}^3 } T[ U \frac{ \partial U }{\partial {y_{1}} }]  U \overline{ v }dx = 0,
 \end{align}
and
\begin{align}\label{condition4}
 \int_{{\mathbb{R}}^3 } T[ U^2 ]  \frac{ \partial U }{\partial {x_{3}} }\overline{ v }dx +2\int_{{\mathbb{R}}^3 } T[ U \frac{ \partial U }{\partial {y_{3}} }]  U \overline{ v }dx = 0.
 \end{align}
In the following, we will prove that  $\overline{v}$ can solve
\begin{align}\label{eqs0}
- \Delta \phi  + \phi - \frac{1}{2}T[ U^2 ] \phi -T[ U\phi ] U =0, \quad \text{in }~ \mathbb{R}^3.
\end{align}
 Define  the constrained  space as follows:
  \begin{align*}
  {\overline{E} _{1}}^+ = \Bigg\{     \phi : \phi \in H^1( \mathbb{R}^3 ),       \int_{{\mathbb{R}}^3}  T[ U^2 ]  &\frac{ \partial U }{\partial {x_{1}} }\phi dx +2\int_{{\mathbb{R}}^3 } T\left[ U \frac{ \partial U }{\partial {y_{1}} }\right]  U \phi dx = 0\quad \text{and} \nonumber
\\[0.02mm]
& \int_{{\mathbb{R}}^3}  T[ U^2 ] \frac{ \partial U }{\partial {x_{3}} }\phi dx +2\int_{{\mathbb{R}}^3 } T\left[ U \frac{ \partial U }{\partial {y_{3}} }\right]  U \phi dx = 0
    \Bigg\}.
 \end{align*}
 To  prove  \eqref{eqs0},
  we provide a  claim firstly.\\
  \textbf{Claim 1}:  $\overline{v}  $ can solve
  \begin{align*}
- \Delta \phi  + \phi - \frac{1}{2}T[ U^2 ] \phi -T[ U\phi ] U =0, \quad \text{in }~   {\overline{E} _{1}}^+.
\end{align*}

Now we prove the   \textbf{Claim 1} .

For any $R>0 $, we let  $\phi \in C_{0}^\infty \big (B_{R}(0) \big) \cap   {\overline{E} _{1}}^+  $ which is even in $y_{2}$.  Then, we can denote $$ \phi_{m} (y)=: \phi (y -  \overline{x}_{1}  )  \in  C_{0}^\infty \big (B_R(  \overline{x}_{1}) \big).$$
 When we insert  $ \phi_{m} (y) =\varphi  $ into \eqref{operatorL} and   combine \eqref{condition0}, \eqref{condition5} and Lemma \ref{lemma1},  we can obtain
 \begin{align} \label{eqs1}
 \int_{\mathbb{R}^3 }  \Big( \nabla  \overline{v} \nabla\phi +  \overline{v} \phi - \frac{1}{2}T[ U^2 ] \phi\overline{v} -T[ U\phi ] U\overline{v} \Big) = 0.
 \end{align}
 Moreover, since $ \overline{v}$ and $ T[U\phi] $ are even in $y_{2}$, we can get   that \eqref{eqs1}   holds for all functions  $\phi  \in C_{0}^\infty \big ( B_{R}(  \overline{x}_{1} \big) \cap  {\overline{E} _{1}}^+  $ which is odd in $y_{2}$ with the symmetric conditions. So, \eqref{eqs1} is valid for all functions $\phi  \in C_{0}^\infty \big ( B_{R}(  \overline{x}_{1} \big) \cap   {\overline{E} _{1}}^+  $.
Density argument indicates that,
\begin{align} \label{eqs2}
 \int_{\mathbb{R}^N }  \Big( \nabla  \overline{v} \nabla\phi +  \overline{v} \phi -\frac{1}{2}T[ U^2 ] \phi\overline{v} -T[ U\phi ] U\overline{v} \Big) = 0,   \quad \text{for all} ~ \phi \in {\overline{E} _{1}}^+.
 \end{align} The  \textbf{Claim 1} is proved.

Putting  \textbf{Claim 1} and the fact that  \eqref{eqs0} is valid  for   $ \phi = \frac{\partial U}{\partial y_{1}} $ and $ \phi = \frac{\partial U}{\partial y_{3}} $  together,  then we can obtain
  \begin{align}
 \int_{\mathbb{R}^3 }  \Big( \nabla  \overline{v} \nabla\phi +  \overline{v} \phi -\frac{1}{2}T[ U^2 ] \phi\overline{v} -T[ U\phi ] U\overline{v} \Big) = 0,   \quad \text{for all} ~ \phi \in  H^1(\mathbb{R}^3),
  \end{align}
that is \eqref{eqs0}. With the help of the Non-degeneracy result for $U$ and the fact that $ \overline{v} $ is even in   $y_{2}$,  we can derive
  \begin{align}  \label{condition4}
  \overline{v} = C_{1}  \frac{\partial U}{\partial y_{1}}+C_{3}  \frac{\partial U}{\partial y_{3}},
\end{align}
 for some constants $ C_{1}$ and $ C_{3}$.   Putting \eqref{condition3}, \eqref{condition4} together, we obtain
\begin{align*}
C_{1}=C_{3}=0.
\end{align*}
So we can get
\begin{align}\label{v0}
\overline{v} =0.
\end{align}
From \eqref{condition5} and \eqref{v0}, we have that
\begin{align}\label{vm}
  \int_{ B_{R}(  \overline{x}_{1} ) }     v_{m}^2  =  o_{m}(1).
\end{align}

From  Lemma \ref{lemma1}, we can know
$
W_{r,t} \leq C  e^{-(1-\delta) | y -   \overline{x}_{1}  | }.
$

From Lemma \ref{lemma1}, the symmetry and Hardy-Littlewood-Sobolev inequality, we have
\begin{align}\label{p1}
 &\frac{1}{4\pi}\int_{ \Omega_{1}}\Big(\int_{\R^3}\frac{ W_{r,t} v_{m}  }{|x-y|}dy\Big)W_{r,t}v_{m}  =\frac{1}{2\pi}\int_{ \Omega_{1}^{+}}\Big(\int_{\R^3}\frac{ W_{r,t} v_{m}  }{|x-y|}dy\Big)W_{r,t}v_{m}\nonumber
\\[1mm]
&\leq C m\int_{ \Omega_{1}^{+}}\Big(\int_{\Omega_{1}^{+}}\frac{ W_{r,t} v_{m}  }{|x-y|}dy\Big)W_{r,t}v_{m}\leq C m\int_{ \Omega_{1}^{+}}\Big(\int_{\Omega_{1}^{+}}\frac{ e^{-(1-\delta) | y -   \overline{x}_{1}  | } v_{m}  }{|x-y|}dy\Big)e^{-(1-\delta) | y -   \overline{x}_{1}  | }v_{m}\nonumber
\\[1mm]
&\leq C m\|e^{-(1-\delta) | y -   \overline{x}_{1}  | } v_{m}\|_{L^{\frac{6}{5}}}\|e^{-(1-\delta) | x -   \overline{x}_{1}  | } v_{m}\|_{L^{\frac{6}{5}}}\nonumber
\\[1mm]
&=C m\Big(  \int_{  \mathbb{R}^3 \setminus  B_{R}( | \overline{x}_{1} | )}  e^{-\frac{6}{5}(1-\delta) | y -   \overline{x}_{1}}  |v_{m}|^\frac{6}{5}+   \int_{    B_{R}( | \overline{x}_{1} | )}  e^{-\frac{6}{5}(1-\delta) | y -   \overline{x}_{1}}  |v_{m}|^\frac{6}{5}\Big)^\frac{5}{3}
\nonumber
\\[1mm]
&\leq C m o_{m}(1)=o(1),
\end{align}
and
\begin{align}\label{p2}
\frac{1}{8\pi}\int_{ \Omega_{1}}\Big(\int_{\R^3}\frac{( W_{r,t})^2 (y)}{|x-y|}dy\Big)  v_{m} ^2 \leq C\int_{ \Omega_{1}^{+}}\Big(\int_{\R^3}\frac{(U_{ \overline{x}_{1}  }+\sum_{i =2}^m U_{ \overline{x}_{i}  } )^2 (y)}{|x-y|}dy\Big)  v_{m} ^2.
\end{align}
By Lemma \ref{lemma4} and \eqref{vm} , we have
\begin{align}\label{p3}
\int_{ \Omega_{1}^{+}}\Big(\int_{\R^3}\frac{(U_{ \overline{x}_{1}  }+\sum_{i =2}^m U_{ \overline{x}_{i}  } )^2 (y)}{|x-y|}dy\Big)  v_{m} ^2&\leq C\Big( e^{- (1-\delta)R  }+\frac{1}{R} \Big)\int_{ \Omega_{1}}v_{m}^2 (x) dx+C \int_{  B_{R}(  \overline{x}_{1}) }v_{m} ^2 (x)dx\nonumber
\\[1mm]
&=C\Big( e^{- (1-\delta)R  }+\frac{1}{R} \Big)\int_{ \Omega_{1}}v_{m}^2 (x) dx+o(1).
\end{align}
By \eqref{p1}, \eqref{p2},  \eqref{p3} and letting $R$ and $m$ sufficiently large and inserting $\varphi=v_{m}$ in \eqref{operatorL}, we can obtain,
\begin{align*}
o_m(1) & =\int_{ \Omega_{1}}     |  \nabla {v_{m} } |^2 + V (|y|) v_{m}^2 -  \frac{1}{2}\int_{ \Omega_{1}}T[( W_{r,t})^2]  v_{m} ^2 -\int_{ \Omega_{1}}T[W_{r,t} v_{m}]W_{r,t}v_{m}    \nonumber
\\[1mm]
&=   \int_{ \Omega_{1}}    |  \nabla {v_{m} } |^2 + V (|y|) v_{m}^2 -  \frac{1}{8\pi}\int_{ \Omega_{1}}\Big(\int_{\R^3}\frac{( W_{r,t})^2 (y)}{|x-y|}dy\Big)  v_{m} ^2 -\frac{1}{4\pi}\int_{ \Omega_{1}}\Big(\int_{\R^3}\frac{ W_{r,t} v_{m}  }{|x-y|}dy\Big)W_{r,t}v_{m}     \nonumber
\\[1mm] & =  \int_{ \Omega_{1}}  \big (   |  \nabla {v_{m} } |^2 + V (|y|) v_{m}^2 + o_{m}(1)+  O_{m}\Big(e^{-(1-\delta)R  }+\frac{1}{R}\Big)  \int_{ \Omega_{1}}  |   {v_{m} } |^2 \nonumber
\\[1mm]
&\geq  \frac 12 \int_{ \Omega_{1}}   |  \nabla {v_{m} } |^2 + V (|y|) v_{m}^2 + o_{m}(1)
\nonumber
\\[1mm]
&=\frac 12+o_{m}(1),
\end{align*}
which is impossible.  The proof of Lemma \ref{lemmainverse} is completed.
\end{proof}

\medskip
Now we prove the following Proposition  which is crucial for the sequel.

\medskip
\begin{proposition}\label{propos2}
For  any $m$ sufficiently large, there is a $C^1 $ map $\Phi: \mathbb{S}_{m}   \rightarrow  \mathbb{E}_{1} $    such that   $ \Phi(r,t)  = \phi_{r,t} (y) \in    \mathbb{E}_{1}$, and
\begin{align}\label{critiacl}
R'(\phi_{r,t} ) = 0, \quad  \text{in } ~ \mathbb{E}_{1}.
\end{align}
Furthermore, there exists a small $\sigma>0$, such that
\begin{align*}
\|\phi_{r,t} \| \leq  \frac{C}{m^{  \frac{2q-1}{2(1-q)} + \sigma } }.
\end{align*}
\end{proposition}
  From Riesz theorem, there exists a $ 1_{m}\in\mathbb{E}_{1}$ satisfying
\begin{align*}
 {  \bf{ l } }(\phi)=({  \bf{ l } }_{m},\phi),\quad  \forall\phi\in\mathbb{E}_{1}.
\end{align*}
In order to apply the contraction mapping theorem to prove that
\eqref{critiacl} is uniquely solvable, we have to estimate
${  \bf{ l } }_{m}$ and $N(\phi)$ respectively.

\begin{lemma}\label{lemmaestimate}
 For all $(r,t) \in \mathbb{S}_{m} $, there exists a small $\sigma>0$, such that
\begin{align}\label{estimatelk}
\|  {  \bf{ l } }_{m}\| \leq    \frac{C}{m^{  \frac{2q-1}{2(1-q)} + \sigma } }.
\end{align}
\end{lemma}

\begin{proof}
  We know   that for any $\phi \in \mathbb{E}_{1}$
\begin{align} \label{definlk}
 \langle  {  \bf{ l } }_{m}, \phi  \rangle &=\sum_{i =1}^m   \int_{  \mathbb{R}^3} \big( V(|y|) - 1      \big)  \big( U_{ \overline{x}_{i}  } + U_{\underline{x}_{i}  }\big)  \phi \nonumber
\\[1mm]
&\quad +\frac{1}{2}  \int_{  \mathbb{R}^3} \Big(   \sum_{i =1}^m T[(U_{ \overline{x}_{i} } )^2]  U_{ \overline{x}_{i}  } + \sum_{i =1}^m T[(U_{\underline{x}_i  }   \big)^2]U_{\underline{x}_{i} }-T[W_{r,t}^2 ]  W_{r,t}   \Big) \phi.
\end{align}
With the help of symmetric property of the functions and an elementary calculation, we can get
\begin{equation}\label{t}
 \frac{1}{|x-\overline{x}_{1}|^{t}}=\frac{1}{|\overline{x}_{1}|^{t}}\Big(1+O\Big(\frac{|x|}{|\overline{x}_{1}|}\Big)\Big), \ x\in B_{\frac{|\overline{x}_{1}|}{2}}(0),\quad \forall\ t>0.
\end{equation}
Assuming that $m$ is even, we
claim that
\begin{align*}
|y-\overline{x}_{1}|\geq (4-\tau)\frac{r}{m}\pi \quad \text{for}\quad  5\leq l \leq \frac{m}{2},\quad y\in \Omega_{l}^{+} ,
\end{align*}
for some small $\tau$ and sufficiently large $m$. In fact, we know

\begin{align*}
|y-\overline{x}_{1}|&\geq|\overline{x}_{1}-\overline{x}_{l}|-|y-\overline{x}_{l}|\nonumber
\\[1mm]
&\geq 2r\sqrt{1-t^{2}}\sin\frac{(l-1)\pi}{m}-4\frac{r}{m}\pi\geq (4-\tau)\frac{r}{m}\pi, \quad \text{if}\quad|y-\overline{x}_{l}|\leq 4\frac{r}{m}\pi,
\end{align*}
and
\begin{align*}
|y-\overline{x}_{1}|\geq|y-\overline{x}_{l}|\geq 4\frac{r}{m}\pi,
\quad \text{if}\quad|y-\overline{x}_{l}|\geq 4\frac{r}{m}\pi.
\end{align*}
By the symmetric property of the functions, we can get
\begin{align}\label{first}
& \sum_{i =1}^m   \int_{\mathbb{R}^3} \big( V(|y|) - 1      \big)  \big( U_{ \overline{x}_{i}  } + U_{\underline{x}_{i}  }\big)  \phi   =2  m   \int_{  \mathbb{R}^3} \big( V(|y|) - 1      \big)   U_{ \overline{x}_{1} }  \phi   \nonumber
\\[0.02mm]
&= 2m \Big( \int_{\cup_{l =1}^4 (\Omega_{l}^+ +\Omega_{l}^-)  } +\int_{\cup_{4<l\leq\frac{m}{2}} (\Omega_{l}^+ +\Omega_{l}^-)  }+\int_{\cup_{\frac{m}{2}<l\leq m} (\Omega_{l}^+ +\Omega_{l}^-)  }\Big)\big( V(|y|) - 1      \big)   U_{ \overline{x}_{1} }  \phi\nonumber
\\[0.02mm]
&\leq C m \Big( \int_{\cup_{l =1}^4 (\Omega_{l}^+ +\Omega_{l}^-)  } +\int_{\cup_{4<l\leq\frac{m}{2}} (\Omega_{l}^+ +\Omega_{l}^-)  }\Big)\big( V(|y|) - 1      \big)   U_{ \overline{x}_{1} }  \phi\nonumber
\\[0.02mm]
&\leq C m \Big( \int_{\cup_{l =1}^4 \Omega_{l}^+   } +\int_{\cup_{4<l\leq\frac{m}{2}} \Omega_{l}^+   }\Big)\big( V(|y|) - 1      \big)   U_{ \overline{x}_{1} }  \phi \nonumber
\\[0.02mm]
& \leq C m \Big( \int_{\cup_{l =1}^4 \Omega_{l}^+ \cap B_{\delta_{1} | \overline{x}_{1} | } (\overline{x}_{1})   } +\int_{\cup_{l =1}^4\Omega_{l}^+ \setminus B_{\delta_{1} | \overline{x}_{1} | } (\overline{x}_{1})  }\Big)\big( V(|y|) - 1      \big)   U_{ \overline{x}_{1} }  \phi
\nonumber
\\[0.02mm]
& \quad+C m \int_{\cup_{4<l\leq\frac{m}{2}} \Omega_{l}^+   }\big( V(|y|) - 1      \big)   U_{ \overline{x}_{1} }  \phi ,
\end{align}
where $ \delta_{1}>0$ is some small enough constant. There exists a positive
constant $C$ such that
\begin{align}\label{C1}
\big( V(|y|) - 1      \big)   U_{ \overline{x}_{1} }& \leq\frac{C}{|y-\overline{x}_{1}+\overline{x}_{1}|^{n}}e^{-|y-\overline{x}_{1}|} \nonumber
\\[0.02mm]
& \leq \frac{C}{r^{q}}e^{-|y-\overline{x}_{1}|},\quad \text{for}\quad y\in\cup_{l =1}^4 \Omega_{l}^+ \cap B_{\delta_{1} | \overline{x}_{1} | } (\overline{x}_{1}),
\end{align}
\begin{align}\label{C2}
\big( V(|y|) - 1      \big)   U_{ \overline{x}_{1} }& \leq C e^{-|y-\overline{x}_{1}|}\nonumber
\\[1mm]
& \leq \frac{C}{r^{q}}e^{-\delta_{1}r},\quad \text{for}\quad y\in\cup_{l =1}^4 \Omega_{l}^+ \setminus B_{\delta_{1} | \overline{x}_{1} |}(\overline{x}_{1}),
\end{align}
\begin{align}\label{C3}
\big( V(|y|) - 1      \big)   U_{ \overline{x}_{1} }& \leq C e^{-|y-\overline{x}_{1}|}\nonumber
\\[1mm]
& \leq e^{-(4-\tau)\frac{r}{m}\pi}   \leq \frac{C}{m^{\frac{1+q}{1-q} }},\quad \text{for}\quad y\in\cup_{4<l\leq\frac{m}{2}} \Omega_{l}^+ .
\end{align}
Then by symmetry, H\"{o}lder inequality and combining \eqref{first}-\eqref{C3} together, we can have
\begin{align}\label{sumC}
\sum_{i =1}^m   \int_{\mathbb{R}^3} \big( V(|y|) - 1      \big)  \big( U_{ \overline{x}_{i}  } + U_{\underline{x}_{i}  }\big)  \phi &\leq C\Big(\frac{1}{(m \ln m)^{\frac{2q-1}{1-q} }}+   \frac{1}{m^{\frac{2q}{1-q} }}\Big)\|\phi\|\nonumber
\\[1mm]
& \leq\frac{C}{m^{\frac{2q-1}{2(1-q)}+\sigma }}\|\phi\|.
\end{align}

The last  inequality is true as $\frac{1}{2}\leq n<1$.

We know that
\begin{align*}
 U_{ \overline{x}_{1}  } + U_{\underline{x}_{1} }  \geq   U_{ \overline{x}_{i}  } + U_{\underline{x}_{i}  } \quad \text{for } ~ y \in  \Omega_{1},
\end{align*}
\begin{align*}
  U_{ \overline{x}_{j}  }    \geq   U_{\underline{x}_{j}  } \quad \text{for } ~ y \in  \Omega_{1}^+.
\end{align*}

From the second  term in \eqref{definlk}, we can get
\begin{align}\label{C11}
 &\frac{1}{2}  \int_{  \mathbb{R}^3} \Big(   \sum_{i =1}^m T[(U_{ \overline{x}_{i} } )^2]  U_{ \overline{x}_{i}  } + \sum_{i =1}^m T[(U_{\underline{x}_{i}  }   \big)^2]U_{\underline{x}_{i} }-T[W_{r,t}^2 ]  W_{r,t}   \Big) \phi\nonumber
 \\[0.02mm]
 & = \frac{1}{8\pi}\int_{\mathbb{R}^3}(\int_{\R^3}\sum_{i =1}^m\frac{  U_{ \overline{x}_{i}  }^2  (y)}{|x-y|}U_{ \overline{x}_{i}  }\phi dxdy+\frac{1}{8\pi}\int_{\mathbb{R}^3}(\int_{\R^3}\sum_{i =1}^m\frac{  U_{ \underline{x}_{i}  }^2  (y)}{|x-y|}U_{ \underline{x}_{i} }\phi dxdy\nonumber
 \\[0.02mm]
 &\quad -\frac{1}{8\pi}
 \int_{\mathbb{R}^3}\Big(\int_{\R^3} \frac{ (    \sum_{i =1}^m \big( U_{ \overline{x}_{i}  } + U_{\underline{x}_{i} }\big))^2  (y)}{|x-y|}dy\Big) \sum_{i =1}^m \big( U_{ \overline{x}_{i}  } + U_{\underline{x}_{i} }\big)  \phi
 \nonumber
 \\[0.02mm]
 & = -\frac{m}{8\pi} \int_{\mathbb{R}^3}\int_{\R^3}\frac{  \big( U_{ \overline{x}_{1}  } \big)^2  (y)}{|x-y|}dy  \big( \sum_{i =2}^m U_{ \overline{x}_{i}  } \big)  \phi-\frac{m}{8\pi} \int_{\mathbb{R}^3}\int_{\R^3}\frac{  \big( U_{ \overline{x}_{1}  } \big)^2  (y)}{|x-y|}dy  \big( \sum_{i =1}^m U_{ \underline{x}_{i}  } \big)  \phi\nonumber
 \\[0.02mm]
 &\quad -\frac{m}{8\pi} \int_{\mathbb{R}^3}\int_{\R^3}\frac{  \big( U_{ \underline{x}_{1}  } \big)^2  (y)}{|x-y|}dy  \big( \sum_{i =2}^m U_{ \underline{x}_{i}  } \big)  \phi-\frac{m}{8\pi} \int_{\mathbb{R}^3}\int_{\R^3}\frac{  \big( U_{ \underline{x}_{1}  } \big)^2  (y)}{|x-y|}dy  \big( \sum_{i =1}^m U_{ \overline{x}_{i}  } \big)  \phi\nonumber
 \\[0.02mm]
 &\quad -\frac{1}{8\pi}
 \int_{\mathbb{R}^3}\Big(\int_{\R^3} \frac{    \sum_{i,j =1\atop i \neq j}^m \big( U_{ \overline{x}_{i}  } + U_{\underline{x}_{i} }\big)\big( U_{ \overline{x}_{j}  } + U_{\underline{x}_{j} }\big)  (y)}{|x-y|}dy\Big) \Big(\sum_{s =1}^m \big( U_{ \overline{x}_{s}  } + U_{\underline{x}_{s} }\big) \Big) \phi\nonumber
 \\[0.02mm]
 & = -\frac{m}{4\pi} \int_{\mathbb{R}^3}\int_{\R^3}\frac{  \big( U_{ \overline{x}_{1}  } \big)^2  (y)}{|x-y|}dy  \big( \sum_{i =2}^m U_{ \overline{x}_{i}  } \big)  \phi-\frac{m}{4\pi} \int_{\mathbb{R}^3}\int_{\R^3}\frac{  \big( U_{ \overline{x}_{1}  } \big)^2  (y)}{|x-y|}dy  \big( \sum_{i =1}^m U_{ \underline{x}_{i}  } \big)  \phi\nonumber
 \\[0.02mm]
 &\quad+O\Big( \frac{1}{8\pi}
 \int_{\mathbb{R}^3}\Big(\int_{\R^3} \frac{    \sum_{i,j =1\atop i \neq j}^m \big( U_{ \overline{x}_{i}  } U_{ \overline{x}_{j}  }+U_{ \overline{x}_{i}  }U_{\underline{x}_{j} } + U_{\underline{x}_{i} } U_{ \overline{x}_{j}  }+U_{\underline{x}_{i} }U_{\underline{x}_{j} }\big)  (y)}{|x-y|}dy\Big)  \phi \Big)\nonumber
 \\[0.02mm]
 & = -\frac{m}{4\pi} \int_{\mathbb{R}^3}\int_{\R^3}\frac{  \big( U_{ \overline{x}_{1}  } \big)^2  (y)}{|x-y|}dy  \big( \sum_{i =2}^m U_{ \overline{x}_{i}  } \big)  \phi-\frac{m}{4\pi} \int_{\mathbb{R}^3}\int_{\R^3}\frac{  \big( U_{ \overline{x}_{1}  } \big)^2  (y)}{|x-y|}dy  \big( \sum_{i =1}^m U_{ \underline{x}_{i}  } \big)  \phi\nonumber
 \\[0.02mm]
 &\quad+O\Big( \frac{1}{8\pi}
 \int_{\mathbb{R}^3}\Big(\int_{\R^3} \frac{    \sum_{i,j =1\atop i \neq j}^m \big( U_{ \overline{x}_{i}  }  U_{ \overline{x}_{j}  }+ U_{ \overline{x}_{i}  } U_{\underline{x}_{j} }\big)  (y)}{|x-y|}dy\Big)  \phi\Big)\nonumber
 \\[0.02mm]
 & = -\frac{m}{4\pi} \int_{\mathbb{R}^3}\int_{\R^3}\frac{  \big( U_{ \overline{x}_{1}  } \big)^2  (y)}{|x-y|}dy  \big( \sum_{i =2}^m U_{ \overline{x}_{i}  } \big)  \phi-\frac{m}{4\pi} \int_{\mathbb{R}^3}\int_{\R^3}\frac{  \big( U_{ \overline{x}_{1}  } \big)^2  (y)}{|x-y|}dy  \big( \sum_{i =1}^m U_{ \underline{x}_{i}  } \big)  \phi\nonumber
 \\[0.02mm]
 &\quad+O\Big(\sum_{i\neq j}^m\Big[ e^{-(1-\delta)|\overline{x}_{i}-\overline{x}_{j}|}+e^{-(1-\delta)|\overline{x}_{i}-\underline{x}_{i}|}\Big]\|\phi\|\| e^{-\delta |\overline{x}_{i}-\overline{x}_{j}|}+e^{-\delta |\overline{x}_{i}-\underline{x}_{i}|} \|\Big)\nonumber
 \\[0.02mm]
 &= -\frac{m}{4\pi} \int_{\mathbb{R}^3}\int_{\R^3}\frac{  \big( U_{ \overline{x}_{1}  } \big)^2  (y)}{|x-y|}dy  \big( \sum_{i =2}^m U_{ \overline{x}_{i}  } \big)  \phi-\frac{m}{4\pi} \int_{\mathbb{R}^3}\int_{\R^3}\frac{  \big( U_{ \overline{x}_{1}  } \big)^2  (y)}{|x-y|}dy  \big( \sum_{i =1}^m U_{ \underline{x}_{i}  } \big)  \phi\nonumber
 \\[0.02mm]
 &\quad+O\Big(\sum_{i\neq j}^m\left[ e^{-(1-\delta)|\overline{x}_{i}-\overline{x}_{j}|}+e^{-(1-\delta)|\overline{x}_{i}-\underline{x}_{i}|}\right]\Big)
 \nonumber
 \\[0.02mm]
 & = -\frac{m}{4\pi} \int_{\mathbb{R}^3}\int_{\R^3}\frac{  \big( U_{ \overline{x}_{1}  } \big)^2  (y)}{|x-y|}dy  \big( \sum_{i =2}^m U_{ \overline{x}_{i}  } \big)  \phi-\frac{m}{4\pi} \int_{\mathbb{R}^3}\int_{\R^3}\frac{  \big( U_{ \overline{x}_{1}  } \big)^2  (y)}{|x-y|}dy  \big( \sum_{i =1}^m U_{ \underline{x}_{i}  } \big)  \phi\nonumber
 \\[0.02mm]
 &\quad +O\Big(\sum_{i\neq j}^m\Big[ e^{-(1-\delta)|\overline{x}_{i}-\overline{x}_{j}|}+e^{-(1-\delta)|\overline{x}_{i}-\underline{x}_{i}|}\Big]\Big)
 \nonumber
 \\[0.02mm]
 & =- \frac{m}{4\pi} D_{i} +O\Big( e^{-2\pi(1-\delta)\sqrt{1-t^{2}}\frac{r}{m}}+e^{-2(1-\delta)rt}\Big),
\end{align}

where
\begin{equation*}
  D_{i}= \int_{\mathbb{R}^3}\int_{\R^3}\frac{  \big( U_{ \overline{x}_{1}  }  \big)^2  (y)}{|x-y|}dy \Big( \sum_{i =2}^m U_{ \overline{x}_{i}  }\Big)  \phi+\int_{\mathbb{R}^3}\int_{\R^3}\frac{  \big( U_{ \overline{x}_{1}  }  \big)^2  (y)}{|x-y|}dy \Big( \sum_{i =1}^m U_{ \underline{x}_{i}  } \Big)  \phi.
\end{equation*}
\par
Next, we will estimate $D_{i}$.
We easily know that for any $y\in\Omega_{j}^{+}$,
 \begin{equation}\label{distance-y}
  |y-\overline{x}_{i}|\geq \frac12|\overline{x}_{j}-\overline{x}_{i}|.
 \end{equation}
In fact,  we can easily  verify that
  \begin{equation*}
    |y-\overline{x}_{j}|\leq|y-\overline{x}_{i}|,
  \end{equation*}
which insures  \eqref{distance-y} if $ |y-\overline{x}_{j}|\geq \frac12|\overline{x}_{j}-\overline{x}_{i}|$. On the other hand, we have
\begin{equation*}
   |y-\overline{x}_{i}|\geq |\overline{x}_{j}-\overline{x}_{i}|-|y-\overline{x}_{j}|\geq\frac12|\overline{x}_{j}-\overline{x}_{i}|.
\end{equation*}
From \eqref{decay} and \eqref{distance-y}, we can get that for any $y\in\Omega_{j}^{+},$
\begin{equation}\label{Ui-decay}
  U_{\overline{x}_{i}}\leq
C e^{-(1-\delta)|y-\overline{x}_{i}|} \leq  C e^{-\frac{1-\delta}{2}|\overline{x}_{j}-\overline{x}_{i}|},
\end{equation}
where $\delta$ is any constant small enough.

By \eqref{sum1i}, \eqref{sum6}, \eqref{Ui-decay} and H\"{o}lder inequality,  we can get

\begin{align}\label{Ew2-decay}
     D_{i}&=\int_{\mathbb{R}^3}\int_{\R^3}\frac{  \big( U_{ \overline{x}_{1}  }  \big)^2  (y)}{|x-y|}dy\Big( \sum_{i =2}^m U_{ \overline{x}_{i}  } \Big)  \phi+\int_{\mathbb{R}^3}\int_{\R^3}\frac{  \big( U_{ \overline{x}_{1}  }  \big)^2  (y)}{|x-y|}dy \Big( \sum_{i =1}^m U_{ \underline{x}_{i}  } \Big)  \phi
    \nonumber
     \\[1mm]
    & \leq C \int_{\mathbb{R}^3}\int_{\Omega_{1}^{+}}\frac{  ( U_{ \overline{x}_{1}  } )^2  (y)}{|x-y|}dy \Big( \sum_{i =2}^m U_{ \overline{x}_{i}  }\Big)  \phi +C\sum_{i,j =2}^m\int_{\mathbb{R}^3}\int_{\Omega_{j}^{+}}\frac{ ( U_{ \overline{x}_{1}  } )^2  (y)}{|x-y|}dy  U_{ \overline{x}_{i}  } \phi\nonumber
     \\[1mm]
    &\quad +C\int_{\mathbb{R}^3}\int_{\Omega_{1}^{+}}\frac{ ( U_{ \overline{x}_{1}  } )^2  (y)}{|x-y|}dy \Big( \sum_{i =1}^m U_{ \underline{x}_{i}  } \Big)  \phi  +C\sum_{j =2\atop i=1 }^m\int_{\mathbb{R}^3}\int_{\Omega_{j}^{+}}\frac{ ( U_{ \overline{x}_{1}  } )^2  (y)}{|x-y|}dy  U_{ \underline{x}_{i}}\phi  \nonumber
     \\[1mm]
    &=C\Big[ \int_{\Omega_{1}^{+}}\int_{\Omega_{1}^{+}}\frac{  ( U_{ \overline{x}_{1}  } )^2  (y)}{|x-y|}dy \Big( \sum_{i =2}^m U_{ \overline{x}_{i}  } \Big)  \phi+\sum_{j =2}^m\int_{\Omega_{j}^{+}}\int_{\Omega_{1}^{+}}\frac{  ( U_{ \overline{x}_{1}  } )^2  (y)}{|x-y|}dy \Big( \sum_{i =2}^m U_{ \overline{x}_{i}  } \Big)  \phi \Big]\nonumber
     \\[1mm]
    &\quad+C\sum_{j =1}^m \int_{\Omega_{j}^{-}}\int_{\Omega_1^{+}}\frac{ ( U_{ \overline{x}_{1}  } )^2  (y)}{|x-y|}dy \Big(  \sum_{i =2}^m U_{ \overline{x}_{i}  } \Big)  \phi+ C\sum_{i =2}^m e^{-\frac{1-\delta}{2}|\overline{x}_{i}-\overline{x}_{1}|}+ C\sum_{i =2}^m e^{-(1-\delta)|\overline{x}_{i}-\overline{x}_{1}|} \nonumber
     \\[1mm]
    &\quad+C\Big( \sum_{j =1}^m \int_{\Omega_{j}^{-}}\int_{\Omega_1^{+}}\frac{ ( U_{ \overline{x}_{1}  } )^2  (y)}{|x-y|}dy\Big(  \sum_{i =1}^m U_{ \underline{x}_{i}  } \Big)  \phi +\sum_{j =1}^m \int_{\Omega_{j}^{+}}\int_{\Omega_1^{+}}\frac{ ( U_{ \overline{x}_{1}  } )^2  (y)}{|x-y|}dy \Big(  \sum_{i =1}^m U_{ \underline{x}_{i}  } \Big)  \phi  \Big) \nonumber
     \\[1mm]
    &\leq C\Big( \sum_{i =2}^m e^{-\frac{1-\delta}{2}|\overline{x}_{i}-\overline{x}_{1}|}+\sum_{i,j =1}^m e^{-\frac{1-\delta}{2}|\overline{x}_{i}-\underline{x}_{j}|}+\sum_{i,j =2 \atop i\neq j}^m\int_{\Omega_{j}^{+}}\int_{\Omega_{1}^{+}}\frac{  ( U_{ \overline{x}_{1}  } )^2  (y)}{|x-y|}dy   U_{ \overline{x}_{i}  }  \phi
    \Big) \nonumber
     \\[1mm]
    &\quad+C\sum_{i =2 }^m\int_{\Omega_{i}^{+}}\int_{\Omega_{1}^{+}}\frac{  ( U_{ \overline{x}_{1}  } )^2  (y)}{|x-y|}dy   U_{ \overline{x}_{i}  }   \phi\nonumber
     \\[1mm]
    &\quad+C\Big(  \sum_{i =1 }^m\int_{\Omega_{i}^{-}}\int_{\Omega_{1}^{+}}\frac{  ( U_{ \overline{x}_{1}  } )^2  (y)}{|x-y|}dy  U_{ \underline{x}_{i}  }   \phi +\sum_{i,j =1 \atop i\neq j}^m\int_{\Omega_{j}^{-}}\int_{\Omega_{1}^{+}}\frac{  ( U_{ \overline{x}_{1}  } )^2  (y)}{|x-y|}dy   U_{ \underline{x}_{i}  }   \phi \Big)\nonumber
     \\[1mm]
    &\leq C\Big( \sum_{i =2 }^m\int_{\Omega_{i}^{+}}\int_{\Omega_{1}^{+}}\frac{  ( U_{ \overline{x}_{1}  } )^2  (y)}{|x-y|}dy  U_{ \overline{x}_{i}  }  \phi+ \sum_{i =1 }^m\int_{\Omega_{i}^{-}}\int_{\Omega_{1}^{+}}\frac{  ( U_{ \overline{x}_{1}  } )^2  (y)}{|x-y|}dy   U_{ \underline{x}_{i}  } \phi \Big)\nonumber
     \\[1mm]
    &\quad+C\Big( \sum_{i,j =1\atop i\neq j}^m e^{-\frac{1-\delta}{2}|\overline{x}_{i}-\overline{x}_{j}|}+\sum_{i,j =1}^m e^{-\frac{1-\delta}{2}|\overline{x}_{i}-\underline{x}_{j}|}   \Big)\nonumber
     \\[1mm]
    &\leq C\Big( \sum_{i =2 }^m\int_{\Omega_{i}^{+}}\int_{\Omega_{1}^{+}}\frac{  ( U_{ \overline{x}_{1}  } )^2  (y)}{|x-y|}dy   U_{ \overline{x}_{i}  }  \phi+ \sum_{i =1 }^m\int_{\Omega_{i}^{-}}\int_{\Omega_{1}^{+}}\frac{  ( U_{ \overline{x}_{1}  } )^2  (y)}{|x-y|}dy  U_{ \underline{x}_{i}  }   \phi \Big)\nonumber
     \\[1mm]
    &\quad+C\Big( \sum_{i,j =1\atop i\neq j}^m e^{-\frac{1-\delta}{2}2r\sqrt{1-t^{2}}\sin\frac{|j-i|\pi}{m}}+\sum_{i,j =1}^m e^{-\frac{1-\delta}{2}2r\sqrt{t^{2}+(1-t^{2})\sin^{2}\frac{|j-i|\pi}{m}}}   \Big)\nonumber
     \\[1mm]
    &\leq C\Big( \sum_{i =2 }^m\int_{\Omega_{i}^{+}}\int_{\Omega_{1}^{+}}\frac{  ( U_{ \overline{x}_{1}  } )^2  (y)}{|x-y|}dy  U_{ \overline{x}_{i}  }   \phi+ \sum_{i =1 }^m\int_{\Omega_{i}^{-}}\int_{\Omega_{1}^{+}}\frac{  ( U_{ \overline{x}_{1}  } )^2  (y)}{|x-y|}dy   U_{ \underline{x}_{i}  }   \phi \Big)\nonumber
     \\[1mm]
    &\quad+C m\Big( \sum_{i=1}^{m-1} e^{-\frac{1-\delta}{2}r\sqrt{1-t^{2}}\frac{i\pi}{m}}+ e^{-(1-\delta)r t}   \Big)\nonumber
     \\[1mm]
    &\leq C\Big( \sum_{i =2 }^m\int_{\Omega_{i}^{+}}\int_{\Omega_{1}^{+}}\frac{  ( U_{ \overline{x}_{1}  } )^2  (y)}{|x-y|}dy  U_{ \overline{x}_{i}  }   \phi+ \sum_{i =1 }^m\int_{\Omega_{i}^{-}}\int_{\Omega_{1}^{+}}\frac{  ( U_{ \overline{x}_{1}  } )^2  (y)}{|x-y|}dy   U_{ \underline{x}_{i}  }   \phi \Big)\nonumber
     \\[1mm]
    &\quad+C m\Big(  e^{-\frac{1-\delta}{2}r\sqrt{1-t^{2}}\frac{\pi}{m}}+ e^{-(1-\delta)r t}   \Big)\nonumber
     \\[1mm]
    &= C\Big( \sum_{i =2 }^m\int_{\Omega_{i}^{+}}\int_{\Omega_{1}^{+}}\frac{  ( U_{ \overline{x}_{1}  } )^2  (y)}{|x-y|}dy  U_{ \overline{x}_{i}  }  \phi+ \sum_{i =1 }^m\int_{\Omega_{i}^{-}}\int_{\Omega_{1}^{+}}\frac{  ( U_{ \overline{x}_{1}  } )^2  (y)}{|x-y|}dy   U_{ \underline{x}_{i}  }  \phi \Big)\nonumber
     \\[1mm]
    &\quad+C\Big(  m e^{-\frac{1-\delta}{4}r\sqrt{1-t^{2}}\frac{\pi}{m}}e^{-\frac{1-\delta}{4}r\sqrt{1-t^{2}}\frac{\pi}{m}}+ m e^{-\frac{1-\delta}{2}r t} e^{-\frac{1-\delta}{2}r t}  \Big)\nonumber
     \\[1mm]
    &=C\Big[ \sum_{i =2}^m\int_{\Omega_{i}^{+}}\int_{\Omega_{1}^{+}}\frac{  ( U_{ \overline{x}_{1}  } )^2  (y)}{|x-y|}dy \Big(  U_{ \overline{x}_{i}  } \Big)  \phi+\sum_{i =1}^m \int_{\Omega_{i}^{-}}\int_{\Omega_1^{+}}\frac{ ( U_{ \overline{x}_{1}  } )^2  (y)}{|x-y|}dy \Big(  U_{ \underline{x}_{i}  } \Big)  \phi  \Big]\nonumber
     \\[1mm]
    &\quad+O\Big( e^{-\frac{1-\delta}{4}r\sqrt{1-t^{2}}\frac{\pi}{m}}+e^{-\frac{1-\delta}{2}rt}\Big)  \nonumber
      \\[1mm]
      &\leq C \sum_{i=2}^m\int_{\Omega_{i}^{+}}
     \int_{B_{\frac{| \overline{x}_{1} - \overline{x}_{i} |}{8}}( \overline{x}_{1} )\cap \Omega_{1}^{+}} e^{-(1-\delta)(2|y-\overline{x}_{1}|+ |x-\overline{x}_{i}|)}\frac{ \phi}{|x-y|}dydx \nonumber
     \\[1mm]
     &\quad
     + O\Big( e^{-\frac{1-\delta}{4}r\sqrt{1-t^{2}}\frac{\pi}{m}}+e^{-\frac{1-\delta}{2}rt}\Big)   \nonumber
     \\[1mm]
   &\quad+ C\sum_{i=1}^m\int_{\Omega_{i}^{-}}
     \int_{B_{\frac{| \overline{x}_{1} - \underline{x}_{i} |}{8}}( \overline{x}_{1} )\cap \Omega_{1}^{+}} e^{-(1-\delta)(2|y-\overline{x}_{1}|+ |x-\underline{x}_{i}|)}\frac{ \phi}{|x-y|}dydx\nonumber
      \\[1mm]
  &\leq C\sum_{i=2}^m\int_{B_{\frac{| \overline{x}_{1} - \overline{x}_{i} |}{8}}( \overline{x}_{1} )\cap\Omega_{i}^{+}}
     \int_{B_{\frac{| \overline{x}_{1} - \overline{x}_{i} |}{8}}( \overline{x}_{1} )\cap \Omega_{1}^{+}} e^{-(1-\delta)(2|y-\overline{x}_{1}|+ |x-\overline{x}_{i}|)}\frac{ \phi}{|x-y|}dydx \nonumber
     \\[1mm]
   &\quad+ C\sum_{i=1}^m\int_{B_{\frac{| \overline{x}_{1} - \underline{x}_{i} |}{8}}( \overline{x}_{1} )\cap\Omega_{i}^{-}}
     \int_{B_{\frac{| \overline{x}_{1} - \underline{x}_{i} |}{8}}( \overline{x}_{1} )\cap \Omega_{1}^{+}} e^{-(1-\delta)(2|y-\overline{x}_{1}|+ |x-\underline{x}_{i}|)}\frac{ \phi}{|x-y|}dydx\nonumber
     \\[1mm]
   &\quad+ O\Big( e^{-\frac{1-\delta}{4}r\sqrt{1-t^{2}}\frac{\pi}{m}}+e^{-\frac{1-\delta}{2}rt}\Big) \nonumber
     \\[1mm]
  &\leq C\sum_{i=2}^m\frac1{|\overline{x}_{1} - \overline{x}_{i} |}\int_{B_{\frac{| \overline{x}_{1} - \overline{x}_{i} |}{8}\cap\Omega_{i}^{+}}( \overline{x}_{1} )}
     \int_{B_{\frac{| \overline{x}_{1} - \overline{x}_{i} |}{8}\cap\Omega_{1}^{+}}( \overline{x}_{1} )} e^{-(1-\delta)(2|y-\overline{x}_{1}|+ |x-\overline{x}_{i}|)}\phi
     dydx\nonumber
     \\[1mm]
     &\quad+C\sum_{i=1}^m\frac1{|\overline{x}_{1} - \underline{x}_{i} |}\int_{B_{\frac{| \overline{x}_{1} - \underline{x}_{i} |}{8}\cap\Omega_{i}^{-}}( \overline{x}_{1} )}
     \int_{B_{\frac{| \overline{x}_{1} - \underline{x}_{i} |}{8}\cap\Omega_{1}^{+}}( \overline{x}_{1} )} e^{-(1-\delta)(2|y-\overline{x}_{1}|+ |x-\underline{x}_{i}|)}\phi
     dydx\nonumber
     \\[1mm]
  &\quad+O\Big( e^{-\frac{1-\delta}{4}r\sqrt{1-t^{2}}\frac{\pi}{m}}+e^{-\frac{1-\delta}{2}rt}\Big) \nonumber
   \\[1mm]
   &\leq C\sum_{i=2}^m\frac1{|\overline{x}_{1} - \overline{x}_{i} |}\|\phi\|+C\frac1{|\overline{x}_{1} - \underline{x}_{1} |}\|\phi\|
    +O\Big( e^{-\frac{1-\delta}{4}r\sqrt{1-t^{2}}\frac{\pi}{m}}+e^{-\frac{1-\delta}{2}rt}\Big) \nonumber
   \\[1mm]
   &\leq C\sum_{i=1}^{m-1}\frac{1}{r\sqrt{1-t^{2}}}\frac{1}{\sin\frac{i\pi}{m}}\|\phi\|+C\frac{1}{rt}\|\phi\|
  +O\Big( e^{-\frac{1-\delta}{4}r\sqrt{1-t^{2}}\frac{\pi}{m}}+e^{-\frac{1-\delta}{2}rt}\Big)  \nonumber
   \\[1mm]
   &\leq C\frac{m\ln m}{r\sqrt{1-t^{2}}}\|\phi\|+O\Big( e^{-\frac{1-\delta}{4}r\sqrt{1-t^{2}}\frac{\pi}{m}}+e^{-\frac{1-\delta}{2}rt}\Big) \nonumber
    \\[1mm]
 &\leq \frac{C}{(m\ln m)^{  \frac{q}{(1-q)}  } }\|\phi\|.
\end{align}

With $\phi(y)=o_r(1)$ when $|y|\geq\frac r2$, from \eqref{definlk}, \eqref{sumC}, \eqref{C11} and \eqref{Ew2-decay}, we can obtain that
\begin{equation*}
  | {  \bf{ l }}(\phi)|\leq    \frac{C}{m^{  \frac{2q-1}{2(1-q)} + \sigma } }\|\phi\|,
\end{equation*}
because of $\frac12\leq q<1$. The proof is completed.
\end{proof}

\begin{lemma}\label{N2}
 There exists a positive constant $C$ satisfying
$$
\|N'(\phi)\|\leq C\|\phi\|^2 \mbox{ and }\|N''(\phi)\|\leq C\|\phi\|.
$$
\end{lemma}

\begin{proof}
From \eqref{N0} and
by direct calculation, there exists $v\in H_{0}^{1}$ satisfying
\begin{align}\label{N'}
&{ \bf{ N } }'(\phi)\nonumber
   \\[1mm] &=  \frac{1}{16\pi}\int_{\R^3}\int_{\R^3}\frac{(                       W_{r,t}+\phi)(y)v(y)(          W_{r,t}+\phi)^2(x)}{|x-y|}+\frac{1}{16\pi}\int_{\R^3}\int_{\R^3}\frac{(          W_{r,t}+\phi)^2(y)(                       W_{r,t}+\phi)(x)v(x)}{|x-y|}\nonumber
   \\[1mm]&
\quad-\frac{1}{8\pi}\int_{\R^3}\int_{\R^3}\frac{( W_{r,t})^2 (y) W_{r,t}(x)v(x)}{|x-y|}-\frac{1}{8\pi}\int_{\R^3}\Big(\int_{\R^3}\frac{( W_{r,t})^2 (y)}{|x-y|}dy\Big)\phi(x) v(x)
\nonumber
   \\[1mm]&
\quad-\frac{1}{8\pi}\int_{\R^3}\Big(\int_{\R^3}\frac{ W_{r,t}(y)v(y)}{|x-y|}dy\Big)W_{r,t}(x)\phi(x)
-\frac{1}{8\pi}\int_{\R^3}\Big(\int_{\R^3}\frac{ W_{r,t}(y)\phi(y)}{|x-y|}dy\Big)W_{r,t}(x)v(x)\nonumber
   \\[1mm]
   &=\frac{1}{8\pi}\int_{\R^3}\int_{\R^3}\frac{                       W_{r,t}(y)v(y)\phi^2(x)}{|x-y|}+\frac{1}{8\pi}\int_{\R^3}\int_{\R^3}\frac{                       \phi(y)v(y)\phi^2(x)}{|x-y|}+\frac{1}{4\pi}\int_{\R^3}\int_{\R^3}\frac{                       \phi(y)v(y)W_{r,t}(x)\phi(x)}{|x-y|},
\end{align}
and
\begin{align}\label{N''}
{ \bf{ N } }''(\phi) &=  \frac{1}{8\pi}\int_{\R^3}\int_{\R^3}\frac{v^2(y)(          W_{r,t}+\phi)^2(x)}{|x-y|}+\frac{1}{4\pi}\int_{\R^3}\int_{\R^3}\frac{(          W_{r,t}+\phi)(y)v(y)(                       W_{r,t}+\phi)(x)v(x)}{|x-y|}\nonumber
   \\[0.02mm]&
\quad-\frac{1}{8\pi}\int_{\R^3}\int_{\R^3}\frac{( W_{r,t})^2 (y) v^2(x)}{|x-y|}
-\frac{1}{4\pi}\int_{\R^3}\Big(\int_{\R^3}\frac{ W_{r,t}(y)v(y)}{|x-y|}dy\Big)W_{r,t}(x)v(x)\nonumber
   \\[0.02mm]
   &=\frac{1}{8\pi}\int_{\R^3}\int_{\R^3}\frac{                      v^2(y)\Big(2W_{r,t}(x)\phi(x)+\phi^2(x)\Big)}{|x-y|}+\frac{1}{4\pi}\int_{\R^3}\int_{\R^3}\frac{ W_{r,t}(y)v(y)                      \phi(x)v(x)}{|x-y|}\nonumber
   \\[0.02mm]&
\quad+\frac{1}{4\pi}\int_{\R^3}\int_{\R^3}\frac{\phi(y)v(y)W_{r,t}(x)v(x)}{|x-y|}+\frac{1}{4\pi}\int_{\R^3}\int_{\R^3}\frac{                       \phi(y)v(y)\phi(x)v(x)}{|x-y|}.
\end{align}
Then, by  Hardy-Littlewood-Sobolev inequality, H\"{o}lder inequality and Sobolev inequality, we can have
$$
\|N'(\phi)\|\leq C(\|\phi\|^2+\|\phi\|^3)\leq C\|\phi\|^2,\quad
\|N''(\phi)\|\leq C(\|\phi\|^2+\|\phi\|)
\leq C\|\phi\|.
$$
\end{proof}
Now we are in a position to prove Proposition \ref{propos2}.\\
{\textit{Proof of Proposition \ref{propos2}     }}:   From Riesz theorem, there exists a $ 1_{m}\in\mathbb{E}_{1}$ satisfying
\begin{align*}
 {  \bf{ l } }(\phi)=({  \bf{ l } }_{m},\phi),\quad  \forall\phi\in\mathbb{E}_{1},\quad \text{and}\quad\|{  \bf{ l } }_{m}\|=\|{  \bf{ l } }\|.
\end{align*}
So, $ \phi $ which is a critical point of R in $ \mathbb{E}_{1}$ is a solution to \eqref{equivalentequat}.
 From Lemma \ref{lemmainverse}, we know that ${\bf{L}}$ is invertible.
Moreover, we can write \eqref{equivalentequat} as follows:
\begin{equation*}
    \phi= M(\phi):=-{\bf{L}}^{-1}\left({  \bf{ l } }_m -{\bf{ N } }'(\phi)\right).
\end{equation*}
 Define \begin{equation*}
  B:=\Big\{\phi\in \mathbb{E}_{1} : \|\phi\|\leq  C\big(\frac{1}{m}\big)^{\frac{2q-1}{2(1-q)} + \sigma} \Big\}.
\end{equation*}
We can claim that $M$ is a contraction map from $B$ to $B$.   In fact, from Lemma \ref{N2}, we can obtain
\begin{equation*}
\|{\bf{ N } }'(\phi)\|\leq C\|\phi\|^2 \mbox{ and }\|{\bf{ N } }''(\phi)\|\leq C\|\phi\|.
\end{equation*}
From Lemmas  \ref{lemmainverse}-- \ref{lemmaestimate}, we can get that
 \begin{equation*}
 \begin{aligned}
     \|M(\phi)\|&\le C\|{  \bf{ l } }_{m}\|+C\|\phi\|^2  \le C\Big(\frac{1}{m}\Big)^{\frac{2q-1}{2(1-q)} + \sigma}.
 \end{aligned}
 \end{equation*}
Moreover, we have
 \begin{equation*}
 \begin{aligned}
     \|M(\phi_{1})-M(\phi_{2})\|&\le  C\|{\bf{ N } }'(\phi_{1})-{\bf{ N } }'(\phi_{2})\|
       \le C\Big(\frac{1}{m}\Big)^{\frac{2q-1}{2(1-q)} + \sigma}
      \|\phi_{1}-\phi_{2}\|\\
      &\le \frac12 \|\phi_{1}-\phi_{2}\|,
 \end{aligned}
 \end{equation*}
 for $m$ large enough. We complete the proof.

\section{Proof of  Theorem \ref{main1}}\label{s3}
In this section, we mainly prove Theorem \ref{main1}.
In order to prove Theorem \ref{main1}, the following proposition is necessary.
\medskip
\begin{proposition} \label{pro2.5}
Assume that $ \Phi(r,t) = \phi_{r,t}(y)$ with $\Phi(r,t) $ be the map which is defined in   Proposition \ref{propos2}.  Define
\begin{align*}
\overline{F}(r,t) =   I\big(W_{r,t}  +     \phi_{r,t}(y)\big),  \qquad  \forall  ~ (r,t)\in  \mathbb{S}_{m}.
 \end{align*}
If  $(r,t)$ is a critical point of $ \overline{F} (r,t) $, then
 \begin{align*}
  u = W_{r,t}  +     \phi_{r,t}(y),
  \end{align*}
   is  a critical point of
 $ I(u)$ in $ H^1(\mathbb{R}^3)$. \qed
\end{proposition} 
\begin{proof}
From \eqref{functional}, we can know that
\begin{align*}
 \overline{F}(r,t)&=\frac{1}{2}\int_{\R^3}|\nabla (W_{r,t}+\phi_{r,t})|^2+V(|x|)(W_{r,t}+\phi_{r,t})^2\\
 &\quad-\frac{1}{32\pi}\int_{\R^3}\int_{\R^3}\frac{(W_{r,t}+\phi_{r,t})^2(y)(W_{r,t}+\phi_{r,t})^2(x)}{|x-y|}.
\end{align*}
The direct calculation leads to
\begin{align}\label{One}
 \overline{F}_{r}=&
 \int_{\R^3}-\triangle\left(W_{r,t}+\phi_{r,t}\right)( \frac{\partial W_{r,t}}{\partial r}+ \frac{\partial  \phi}{\partial r})+V(|x|)(W_{r,t}+\phi_{r,t})( \frac{\partial W_{r,t}}{\partial r}+ \frac{\partial  \phi}{\partial r})dx\nonumber
\\[0.02mm]
&-\frac{1}{8\pi}\int_{\R^3}\int_{\R^3}\frac{(W_{r,t}+\phi_{r,t})^2(y)(W_{r,t}+\phi_{r,t})( \frac{\partial W_{r,t}}{\partial r}+ \frac{\partial  \phi}{\partial r})(x)}{|x-y|}.
\end{align}
Similarly,
\begin{align}\label{Two}
 \overline{F}_{t}=&
 \int_{\R^3}-\Delta\left(W_{r,t}+\phi_{r,t}\right)( \frac{\partial W_{r,t}}{\partial t}+ \frac{\partial  \phi}{\partial t})+V(|x|)(W_{r,t}+\phi_{r,t})( \frac{\partial W_{r,t}}{\partial t}+ \frac{\partial  \phi}{\partial t})dx\nonumber
\\[0.02mm]
&-\frac{1}{8\pi}\int_{\R^3}\int_{\R^3}\frac{(W_{r,t}+\phi_{r,t})^2(y)\big(W_{r,t}+\phi_{r,t}\big)( \frac{\partial W_{r,t}}{\partial t}+ \frac{\partial  \phi}{\partial t})(x)}{|x-y|}.
\end{align}
From \eqref{SpaceE}, we know that
\begin{align*}
\mathbb{E}_{1}^{**}=
\text{span} \left\{ \frac{\partial U_{\overline{x}_{j}}}{\partial r},\frac{\partial U_{\underline{x}_{j}}}{\partial r},\frac{\partial U_{\overline{x}_{j}}}{\partial t},\text{and} \frac{\partial U_{\underline{x}_{j}}}{\partial t},j=1,2,3 \right\}.
\end{align*}
Since $R'(\phi_{r,t})=0 $ in $\mathbb{E}_{1}$ and $\phi_{r,t}\in\mathbb{E}_{1} $, we have
\begin{align}\label{R'}
\langle R', \phi_{r,t} \rangle =&\int_{\R^3}\left[-\triangle(W_{r,t}+\phi_{r,t})+V(|x|)(W_{r,t}+\phi_{r,t})\right]\phi_{r,t}dx\nonumber
\\[0.02mm]
& -\frac{1}{8\pi}\int_{\R^3}\int_{\R^3}\frac{(W_{r,t}+\phi_{r,t})^2(y)(W_{r,t}+\phi_{r,t})\phi_{r,t}(x)}{|x-y|}=0.
\end{align}
So, we can get that
\begin{align*}
-\Delta(W_{r,t}+\phi_{r,t})+V(|x|)(W_{r,t}+\phi_{r,t})-\frac{1}{8\pi}\int_{\R^3}\frac{(W_{r,t}+\phi_{r,t})^2(y)}{|x-y|}dy(W_{r,t}+\phi_{r,t})\in\mathbb{E}_{1}^{\bot}.
\end{align*}
Then, there exists some constants $ C_{i},C_{j},C_{p},C_{q},$ satisfying
\begin{align}\label{bot}
&-\Delta(W_{r,t}+\phi_{r,t})+V(|x|)(W_{r,t}+\phi_{r,t})-\frac{1}{8\pi}\int_{\R^3}\frac{(W_{r,t}+\phi_{r,t})^2(y)}{|x-y|}dy(W_{r,t}+\phi_{r,t})\nonumber
\\[0.02mm]
&=\sum_{i=1}^m C_{i} \frac{\partial U_{\overline{x}_{i}}}{\partial r}+\sum_{j=1}^m C_{j}\frac{\partial U_{\underline{x}_{j}}}{\partial r}+\sum_{p=1}^m C_{p}\frac{\partial U_{\overline{x}_{p}}}{\partial t}+\sum_{q=1}^m C_{q}\frac{\partial U_{\underline{x}_{q}}}{\partial t}.
\end{align}
So we have
\begin{align*}
&\overline{F}_{r}+\overline{F}_{t}
\nonumber
\\[0.02mm]
&=\Big \langle \sum_{i=1}^m C_{i} \frac{\partial U_{\overline{x}_{i}}}{\partial r}+\sum_{j=1}^m C_{j}\frac{\partial U_{\underline{x}_{j}}}{\partial r}+\sum_{p=1}^m C_{p}\frac{\partial U_{\overline{x}_{p}}}{\partial t}+\sum_{q=1}^m C_{q}\frac{\partial U_{\underline{x}_{q}}}{\partial t},\frac{\partial W_{r,t}}{\partial r}+ \frac{\partial W_{r,t}}{\partial t}+\frac{\partial  \phi}{\partial r}+ \frac{\partial  \phi}{\partial t}\Big \rangle\nonumber
\\[0.02mm]
&=\Big\langle \sum_{i=1}^m C_{i} \frac{\partial U_{\overline{x}_{i}}}{\partial r}+\sum_{j=1}^m C_{j}\frac{\partial U_{\underline{x}_{j}}}{\partial r}+\sum_{p=1}^m C_{p}\frac{\partial U_{\overline{x}_{p}}}{\partial t}+\sum_{q=1}^m C_{q}\frac{\partial U_{\underline{x}_{q}}}{\partial t},\frac{\partial W_{r,t}}{\partial r}+ \frac{\partial W_{r,t}}{\partial t}\Big \rangle\nonumber
\\[0.02mm]
&\quad +\Big \langle \sum_{i=1}^m C_{i} \frac{\partial U_{\overline{x}_{i}}}{\partial r}+\sum_{j=1}^m C_{j}\frac{\partial U_{\underline{x}_{j}}}{\partial r}+\sum_{p=1}^m C_{p}\frac{\partial U_{\overline{x}_{p}}}{\partial t}+\sum_{q=1}^m C_{q}\frac{\partial U_{\underline{x}_{q}}}{\partial t},\frac{\partial  \phi}{\partial r}+ \frac{\partial  \phi}{\partial t}\Big \rangle
\nonumber
\\[0.02mm]
&=0.
\end{align*}
Noting that $\phi_{r,t}\in\mathbb{E}_{1}$,
 we have
\begin{align*}
\Big \langle\frac{\partial U_{\overline{x}_{i}}}{\partial r}, \phi \Big \rangle=\Big \langle\frac{\partial U_{\underline{x}_{j}}}{\partial r}, \phi \Big \rangle=\Big \langle\frac{\partial U_{\overline{x}_{p}}}{\partial t}, \phi \Big \rangle=\Big \langle\frac{\partial U_{\underline{x}_{q}}}{\partial t}, \phi \Big \rangle
=0.
\end{align*}
By the direct calculations, we have
$$
\Big \langle\frac{\partial U_{\overline{x}_{i}}}{\partial r},\frac{\partial \phi}{\partial r}  \Big \rangle=-\Big \langle\frac{\partial^{2} U_{\overline{x}_{i}}}{\partial r^{2}}, \phi \Big \rangle, \quad
\Big \langle\frac{\partial U_{\underline{x}_{i}}}{\partial r}, \frac{\partial \phi}{\partial r}\Big \rangle=-\Big \langle\frac{\partial^{2} U_{\underline{x}_{i}}}{\partial r^{2}}, \phi \Big \rangle ,$$
$$\Big \langle\frac{\partial U_{\overline{x}_{i}}}{\partial t}, \frac{\partial \phi}{\partial t} \Big \rangle=-\Big \langle\frac{\partial^{2} U_{\overline{x}_{i}}}{\partial t^{2}}, \phi \Big \rangle,\quad
\Big \langle\frac{\partial U_{\underline{x}_{i}}}{\partial t}, \frac{\partial \phi}{\partial t} \Big \rangle=-\Big \langle\frac{\partial^{2} U_{\underline{x}_{i}}}{\partial t^{2}}, \phi \Big \rangle,
$$
$$
\Big \langle\frac{\partial U_{\overline{x}_{i}}}{\partial t},\frac{\partial \phi}{\partial r}  \Big \rangle=-\Big \langle\frac{\partial^{2} U_{\overline{x}_{i}}}{\partial r \partial t}, \phi \Big \rangle, \quad
\Big \langle\frac{\partial U_{\underline{x}_{i}}}{\partial t}, \frac{\partial \phi}{\partial r} \Big \rangle=-\Big \langle\frac{\partial^{2} U_{\underline{x}_{i}}}{\partial r\partial t}, \phi \Big \rangle ,$$
$$\Big \langle\frac{\partial U_{\overline{x}_{i}}}{\partial r}, \frac{\partial \phi}{\partial t} \Big \rangle=-\Big \langle\frac{\partial^{2} U_{\overline{x}_{i}}}{\partial r\partial t}, \phi \Big \rangle,\quad
\Big \langle\frac{\partial U_{\underline{x}_{i}}}{\partial r}, \frac{\partial \phi}{\partial t} \Big \rangle=-\Big \langle\frac{\partial^{2} U_{\underline{x}_{i}}}{\partial r\partial t}, \phi \Big \rangle,
$$
which implies that
$$
\Big| \Big \langle\frac{\partial U_{\overline{x}_{i}}}{\partial r},\frac{\partial \phi}{\partial r}  \Big \rangle\Big|=\Big| -\Big \langle\frac{\partial^{2} U_{\overline{x}_{i}}}{\partial r^{2}}, \phi \Big \rangle\Big| =O(  \| \phi\|)=O\Big(\Big(\frac{1}{m}\Big)^{\frac{2n-1}{2(1-n)}+\sigma}\Big).
$$
Similarly, we can prove that the others have the same estimations. So we can get
\begin{align*}
&\overline{F}_{r}+\overline{F}_{t}
\nonumber
\\[0.02mm]
&=\Big \langle \sum_{i=1}^m C_{i} \frac{\partial U_{\overline{x}_{i}}}{\partial r}+\sum_{j=1}^m C_{j}\frac{\partial U_{\underline{x}_{j}}}{\partial r}+\sum_{p=1}^m C_{p}\frac{\partial U_{\overline{x}_{p}}}{\partial t}+\sum_{q=1}^m C_{q}\frac{\partial U_{\underline{x}_{q}}}{\partial t},\frac{\partial W_{r,t}}{\partial r}+ \frac{\partial W_{r,t}}{\partial t}\Big \rangle +O\Big(\big(\frac{1}{m}\big
)^{\frac{2n-1}{2(1-n)}+\sigma}\Big) \nonumber
\\[0.02mm]
&=\sum_{i=1}^m C_{i} \Big(\frac{\partial U_{\overline{x}_{i}}}{\partial r}\Big)^{2}+\sum_{j=1}^m C_{j}\Big(\frac{\partial U_{\underline{x}_{j}}}{\partial r}\Big)^{2}+\sum_{p=1}^m C_{p}\Big(\frac{\partial U_{\overline{x}_{p}}}{\partial t}\Big)^{2}+\sum_{q=1}^m C_{q}\Big(\frac{\partial U_{\underline{x}_{q}}}{\partial t}\Big)^{2}+O\Big(\big(\frac{1}{m}\big
)^{\frac{2n-1}{2(1-n)}+\sigma}\Big)
\\[0.02mm]
&=0.
\end{align*}
So, we can get $C_{i},C_{j},C_{p},C_{q}=0 $. Consequently, $I'(W_{r,t}  +\phi_{r,t})=0 $ in $H^1(\mathbb{R}^3)$ which shows that $W_{r,t}  +\phi_{r,t}(y) $ is  a critical point of
 $ I(u)$ in $ H^1(\mathbb{R}^3)$. We complete the proof.
\end{proof}

Now we will prove the Theorem \ref{main1}.

\medskip
{\textit{Proof of Theorem \ref{main1}  } }:     By Proposition \ref{func}, we need to prove
that there  exists $(r_m,t_m)\in \mathbb{S}_{m}$, which is a critical point of $\overline{F}(r,t)$.

Indeed, from Proposition  \ref{func},   we have
\begin{align*}
\overline{F}(r,t) & =   I(W_{r,t} ) + {  \bf{ l } }  (\phi_{r,t}) + \frac 12  \langle{\bf{L}} \phi_{r,t},  \phi_{r,t}\rangle  -  {  \bf{ N } }(\phi_{r,t}) \nonumber
\\[0.02mm]
& =  I(W_{r,t} ) +  O_{m}(\|{\bf{l}}_{m} \|     \| \phi_{r,t}\| +      \|\phi_{r,t}\|^2  ) =  I(W_{r,t} )  +     O_{m}\Big(\frac{1}{m^{\frac{2q-1}{1-q}+\sigma}}\Big)\nonumber
\\[0.02mm]
& =
    m\Big(     \frac{A_{2}}{16\pi}+ \frac{A_{1}}{r^q}  -  \frac{A_{1}^2}{16\pi^2\sqrt{1-t^{2}}}\frac{m\ln m}{r}-\frac{m}{16\pi^2}\frac{1 }{r\sqrt{1-t^{2}}}\ln\frac{\pi}{t}A_{1}^{2}   \Big)+m O_{m}(\frac{1}{r^{q+\tau}})  \nonumber
   \\[0.02mm]
  &\quad+m O_{m}\Big(\frac{m^{2}}{ r^{2}(1-t^{2})}\Big)+\frac{m^{2} }{r\sqrt{1-t^{2}}}\ln\frac{\pi}{t}O_{m} \big(t^{2} |\ln t|^{-1}\big)+O_{m}\Big(\frac{1}{m^{\frac{2q-1}{1-q}+\sigma}}\Big),
\end{align*}
where $A_{1}, A_{2}$ are  constants  in \eqref{A1A2}.

Define
$$ \overline{F}_{1} ( r, t) =
   \frac{A_{2}}{16\pi}+ \frac{A_{1}}{r^q}  -  \frac{A_{1}^2}{16\pi^2\sqrt{1-t^{2}}}\frac{m\ln m}{r}-\frac{m}{16\pi^2}\frac{1 }{r\sqrt{1-t^{2}}}\ln\frac{\pi}{t}A_{1}^{2} .$$
So, we consider the  following system
\begin{align}\label{aot}
 \begin{cases}
 -A_1 \frac q{r^{q+1}}   +     \frac {A_1^2} {16\pi^2} \frac { m \ln m} {\sqrt{1-t^2}r^2}    +     \frac {A_1^2} {16\pi^2}\frac { m} {\sqrt{1-t^2}r^2}  \ln{\frac \pi t} =0,    \\[4mm]
     - \frac {t m \ln m} { r(1-t^2)^{\frac 32} }  - \frac{m t} {r (1-t^2)^{\frac 32} }  \ln{\frac \pi t} +\frac{m} {tr\sqrt{1-t^2} } = 0,
 \end{cases}
 \end{align}
ie,
 \begin{align}
 \begin{cases}
 -A_1 \frac q{r^{q+1}}   +     \frac {A_1^2} {16\pi^2} \frac { m \ln m} {\sqrt{1-t^2}r^2}    +     \frac {A_1^2} {16\pi^2}\frac { m} {\sqrt{1-t^2}r^2}  \ln{\frac \pi t} =0,  \\[4mm]
     - \frac {t m \ln m} { (1-t^2) }  - \frac{mt} { (1-t^2) }  \ln{\frac \pi t} +\frac{m} {t } = 0,
     \end{cases}
     \end{align}
ie,
 \begin{align}
 \begin{cases}
 -A_1 \frac q{r^{q-1}}   +      \frac {A_1^2} {16\pi^2}\frac { m \ln m} {\sqrt{1-t^2}}    +    \frac {A_1^2} {16\pi^2} \frac { m} {\sqrt{1-t^2}}  \ln{\frac \pi t} =0,   \\[4mm]
     - \frac {t  \ln m} { (1-t^2) }  - \frac{t} { (1-t^2) }  \ln{\frac \pi t} +\frac{1} {t } = 0.
     \end{cases}
     \end{align}
The problem is equivalent to the following fixed point
\begin{align}
{\bf A}(t)&= \frac { \ln m} { (1-t^2) }  +  \frac{1} { (1-t^2) }  \ln{\frac \pi t},  \text{where}~ {\bf A}(t):=  \frac{1} {t^2 }.
\end{align}
Then for any $(r,t)\in \mathbb{S}_{m}$, define
\begin{align}
  {\bf a} (t) = &   {\bf A}^{-1}  \Bigg( \frac { \ln m} { (1-t^2) }  +  \frac{1} { (1-t^2) }  \ln{\frac \pi t}\Bigg)
  \\ = & \frac 1{  \Big( \frac { \ln m} { (1-t^2) }  +  \frac{1} { (1-t^2) }  \ln{\frac \pi t}\Big)^{\frac12}  }
    \\ = &   \frac {(1-t^2)^{\frac12} } {  \Big( \ln m +   (\ln \pi - \ln t)\Big)^{\frac12}  }
     \\ = &   \frac {(1-t^2)^{\frac12} } {  \Big( \ln m +    \ln \pi  \Big)^{\frac12}  }    \frac 1{(1-\frac{\ln t} {\ln m +   \ln \pi }  )^{\frac 12} }.
 \end{align}

   By the direct calculations, we have
\begin{align}\label{b2}
 | \,  {\bf a}   (t_{1})  -  {\bf a}  (t_{2})\,    |\nonumber
\\[0.02mm]
  &=
\Big|\frac {(1-t_{1}^2)^{\frac12} } {  \Big( \ln m +    \ln \pi  \Big)^{\frac12}  }    \frac 1{(1-\frac{\ln t_{1}} {\ln m +    \ln \pi }  )^{\frac 12} }-\frac {(1-t_{2}^2)^{\frac12} } {  \Big( \ln m +    \ln \pi  \Big)^{\frac12}  }    \frac 1{(1-\frac{\ln t_{2}} {\ln m +    \ln \pi }  )^{\frac 12} }\Big|\nonumber
\\[0.02mm]
&=\frac{1+o(1)}{  \Big( \ln m +   \ln \pi  \Big)^{\frac12}   } \Big|(1-t_{1}^{2})^{\frac{1}{2}}-(1-t_{2}^{2}) ^{\frac{1}{2}} \Big|\nonumber
\\[0.02mm]
&=\frac{1+o(1)}{  \Big( \ln m +    \ln \pi  \Big)^{\frac12}   } \Big| t_{1} -t_{2} \Big|O(t)
\nonumber
\\[0.02mm]
&=O\Big(\frac{1}{\ln m}\Big)|t_{1}-t_{2}|=\theta|t_{1}-t_{2}|,
\end{align}
where $ 0<\theta<1$, and we have $t\sim  (\ln m)^{-\frac{1}{2}}$ by Remark \ref{remark6}.
With the help of the  Contraction Mapping principle, we can have that  there is a  fixed point $( \bar{r}_{m}, \bar{t}_{m} )  \in \bar {\mathbb{S} }_{m}$. In other words,  $ \overline{F}_{1} ( r,t) $ have a critical point  $( \bar{r}_{m}, \bar{t}_{m} )  \in \bar {\mathbb{S} }_{m}$.

Define
         \[    {   \bf B_{2}     } ( r, t )    = \left(
\begin{array}{cccc}
 \overline{F}_{1,rr} & \overline{F}_{1,r t}
\\[0.02mm]
\overline{F}_{1,r t}& \overline{F}_{1,t t}
\end{array}
\right). \]
By computing, we can get  \[      \overline{F}_{1,r r} \big| _{( r, t )   =( \bar{r}_m, \bar{t}_{m} )  }  > 0, \quad  \overline{F}_{1,t t} \big| _{( r, t )   =( \bar{r}_{m}, \bar{t}_{m} )  }  < 0 , \quad  \overline{F}_{1,r t} \big| _{( r, t )   =( \bar{r}_m, \bar{t}_{m} )  }   >0 ,  \]
  and
  \[     \overline{F}_{1,r r} \times  \overline{F}_{1,t t} \big| _{( r, t )   =( \bar{r}_{m}, \bar{t}_{m} )  }  - \overline{F}_{1,r t}^2 \big| _{( r, t )   =( \bar{r}_{m}, \bar{t}_{m} )  }  < 0.   \]
 Consequently, we can obtain that  $ ( \bar{r}_m, \bar{t}_m ) $ is a maximum point of $\overline{F}_{1} ( r,t) $.
  So the
maximum of  $\overline{F}_{1} ( r,t) $ in  $ {\mathbb{S} }_{m}$ can be achieved.

 So, for the function $\overline{F}(r,t)$,   we can seek a maximum  point $(r_{m}, t_{m} ) $ that is an interior point of $   {\mathbb{S} }_{m}$.  Thus, $(r_{m}, t_{m} ) $ is a critical point of $ \overline{F}(r,t)$.
Then
$$  W_{r_{m},t_{m}}  +     \phi_{r_{m},t_{m}}(y),$$ is a critical point of $I(u)$. The proof of  Theorem \ref{main1} is completed.
\qed
 \vspace{2mm}

\begin{remark} \label{remark6}
By \eqref{aot}, we have
\begin{align*}
  \overline{F}_{1,r}(r,t)  =  -   A_{1}    \frac{q }{r^{q+1 }}  +  \frac{A_{1}^2}{16\pi^2\sqrt{1-t^{2}}}\frac{m\ln m}{r^{2 }} +\frac{m A_{1}^2}{16\pi^2 r^{2 }\sqrt{1-t^{2}}}\ln\frac{\pi}{t}  =0.
\end{align*}
Since the second term is much bigger than the third term, we only keep the first term and the the second term. By \eqref{aot}, we also have
\begin{align*}
 \overline{F}_{1,t}(r,t)  =  -  \frac{t A_{1}^2}{16\pi^2 (1-t^{2})^{\frac{3}{2}}}\frac{m\ln m}{r}-\frac{t A_{1}^2}{16\pi^2 (1-t^{2})^{\frac{3}{2}}}\frac{m\ln\frac{\pi}{t}}{r}+\frac{ m A_{1}^2}{16r\pi^2 \sqrt{1-t^{2}}}\frac{1}{t}.
\end{align*}
For the simplicity of getting $t$, we only keep the first term and the the third term. Then, we consider the following system
 \begin{align}
 \begin{cases} \label{system1}
  \overline{F}_{1,r}(r,t)  =    -   A_{1}    \frac{q }{r^{q+1 }}  +  \frac{A_{1}^2}{16\pi^2\sqrt{1-t^{2}}}\frac{m\ln m}{r^{2 }}\approx 0,
 \\[2mm]
 \overline{F}_{1,t}(r,t)  =  -  \frac{t A_{1}^2}{16\pi^2 (1-t^{2})^{\frac{3}{2}}}\frac{m\ln m}{r}+\frac{ m A_{1}^2}{16r\pi^2 \sqrt{1-t^{2}} }\frac{1}{t} \approx 0.
\end{cases}
\end{align}
And we can get
$$r=\Big(\frac{A_{1}}{16q( \pi)^{2}    }\Big)^{\frac{1}{1-q}}\big(1+o(1)\big)(m \ln m)^{\frac{1}{1-q}}, t=\Big(1+o(1)  \Big) (  \ln m)^{-\frac{1}{2}}.$$
So we assume that ~$(r,t)$~satisfies~\eqref{H2}~.
\end{remark}

 \vspace{3mm}
 
 \appendix
\section{ Some known results and technical estimates}\label{sa}
Firstly, we  provide  some essential estimates.
\begin{lemma}\label{lemma0}
 For any $u,v,\mu,\varphi\in H^1(\R^3)$, we have
\begin{equation}\label{Hardy-LS}
   \int_{\R^3}\int_{\R^3}\frac{u(y)v(y)\omega(x)\upsilon(x)}{|x-y|} \leq C \|u\| \|v\| \|\omega\| \|\upsilon\|.
\end{equation}
\end{lemma}

 \begin{proof}
 By  Hardy-Littlewood-Sobolev inequality with $\frac{1}{\lambda}+\frac{1}{s}+\frac13=2$, H\"{o}lder inequality and Sobolev inequality, we can get
\begin{equation}\label{Hardy-LS}
   \int_{\R^3}\int_{\R^3}\frac{u(y)v(y)\omega(x)\upsilon(x)}{|x-y|}\leq \|uv\|_{L^\lambda(\R^3)}\|\omega\upsilon\|_{L^s(\R^3)} \leq C \|u\| \|v\| \|\omega\| \|\upsilon\|.
\end{equation}
\end{proof}

\medskip
\begin{lemma}\label{lemma1}
For  $r,t $  being  the parameters   in \eqref{overunder}  and  any $\delta \in (0,1] $, there  exists $ C>0$ such that
\begin{align*}
\sum_{i =2}^m  U_{ \overline{x}_{j}} (y)& \leq C e^{- \delta \sqrt{1-t^2} r\frac{\pi}{m}}e^{   -(1-\delta) |y-\overline{x}_{1}| },  \quad \text{for all} ~  y \in \Omega_{1}^+,
\end{align*}
\begin{align*}
\sum_{i =2}^m  U_{ \underline{x}_{j}} (y)& \leq C e^{- \delta \sqrt{1-t^2} r\frac{\pi}{m}} e^{   -(1-\delta) |y-\overline{x}_{1}| },  \quad \text{for all} ~  y \in \Omega_{1}^+,
\end{align*}
and
\begin{align*}
 U_{ \underline{x}_{1}} (y)\leq C   e^{ -   \delta  t r}   e^{-(1-\delta) | y -   \overline{x}_{1}  | },   \quad \text{for all} ~  y \in \Omega_{1}^+.
\end{align*}
\end{lemma}

\begin{proof}
 Since the proof is just similar to  Lemma A.1 in \cite{wei&yan}, here we omit it.
\end{proof}
\medskip

In this appendices, we assume  $(r,t)  \in \mathbb{S}_{m} $, where    $\mathbb{S}_{m}$ is defined in \eqref{H2}.

\medskip

\begin{lemma}\label{lemma2}
  The following expansion is valid
   \begin{equation}\label{interact}
      \sum_{i=2}^m \int_{\R^3}\Psi_{ U_{\overline{x}_{1}}} U_{\overline{x}_{i}}^2=\frac{A_1^2}{8\pi^2}\frac{m\ln m}{r\sqrt{1-t^{2}}}+ O\Big(\frac{m^2}{r^{2}(1-t^{2})} \Big),
   \end{equation}
     where $A_{1}$ is defined in \eqref{A1A2}.

\end{lemma}

\begin{proof}
The direct calculation shows that
\begin{equation}\label{calculation}
\begin{aligned}
    \sum_{i=2}^m \int_{\R^3}\Psi_{ U_{\overline{x}_{1}}} U_{\overline{x}_{i}}^2&= \frac{1}{8\pi}\sum_{i=2}^m\int_{\R^3}\int_{\R^3}\frac{U_{\overline{x}_{1}}^2(y)U_{\overline{x}_{i}}^2(x)}{|x-y|}\\
    &=\frac{1}{8\pi}\sum_{i=2}^m\int_{\R^3}U^2(x)\int_{\R^3}U^2(y)\frac{1}{|x-y+(\overline{x}_{1}-\overline{x}_{i})|}\\
    &=\frac{1}{8\pi}\sum_{i=2}^m\frac{1}{|\overline{x}_{1}-\overline{x}_{i}|}\int_{\R^3}U^2(x)\int_{\R^3}U^2(y)
    +  O\Big(\sum_{i=2}^m\frac{1}{|\overline{x}_{1}-\overline{x}_{i}|^2}\Big).
\end{aligned}
\end{equation}

Observe that
\begin{equation}\label{distance}
\begin{aligned}
   \sum_{i=2}^m\frac{1}{|\overline{x}_{1}-\overline{x}_{i}|}&=\frac{1}{2r\sqrt{1-t^2}} \sum_{i=1}^{m-1}\frac{1}{ \sin\frac{i\pi}{m} }.
\end{aligned}
 \end{equation}
 It has been checked that
\begin{equation*}
\int_{\frac32}^{m-\frac{3}{2}}\frac{1}{\sin\frac{x\pi}{m}} \leq \sum_{i=1}^{m-1}\frac{1}{\sin\frac{i\pi}{m}}\leq \int_{\frac12}^{m-\frac{1}{2}}\frac{1}{\sin\frac{x\pi}{m}},
\end{equation*}
and
\begin{equation*}
  \lim_{m\rightarrow\infty}\frac{1}{m\ln m}\int_{\frac32}^{m-\frac{3}{2}}\frac{1}{\sin\frac{x\pi}{m}}= \lim_{m\rightarrow\infty}\frac{1}{m\log m}\int_{\frac12}^{m-\frac{1}{2}}\frac{1}{\sin\frac{x\pi}{m}}
  =\frac2\pi.
\end{equation*}
So, we can get
\begin{equation}\label{sum1i}
  \sum_{i=2}^m\frac{1}{|\overline{x}_{1}-\overline{x}_{i}|}=\frac{m \ln m}{\pi r\sqrt{1-t^2}}+o_{m}(1).
\end{equation}
From the definitions $ \overline{x}_{j}, \underline{x}_{j}$, we have
\begin{align*}
   |  \overline{x}_{i}- \overline{x}_{1}|^2  = 4r^2  ( 1- t^2)    \sin^2 { \frac{ ( i-1 )\pi  }{m}    }.
\end{align*}
Similarly, we obtain
\begin{equation}\label{2}
  O\Big(\sum_{i=2}^m\frac{1}{|\overline{x}_{1}-\overline{x}_{i}|^2}\Big)=  O\Big(\frac{m^{2}}{r^{2}(1-t^{2})} \Big).
\end{equation}
From \eqref{calculation}, \eqref{sum1i} and \eqref{2}, we can get \eqref{interact} .
\end{proof}
\begin{lemma}\label{lemma3}
There is a small enough constant $\delta>0,\sigma>0$   such that the following expansions are valid
\begin{align}  \label{suma}
 \int_{  \mathbb{R}^3}  U_{     \overline{x}_{1}  }   U_{\underline{x}_{1}  }    =   O_{m}( e^{  - 2(1-\delta) r t} ),
\end{align}
\begin{align}  \label{sum2}
\sum_{i =2}^m  \int_{  \mathbb{R}^3}  U_{     \overline{x}_{1}  }   U_{\overline{x}_{i}  }    =     O_{m}(     e^{  - 2\pi(1-\delta)  \sqrt{1-t^2}  \frac r m                }     +
e^{  - 2(1-\delta)(1+\sigma) \pi  \sqrt{1-t^2}  \frac r m                }  ),
\end{align}
and
\begin{align}\label{sum3}
   \sum_{i =1}^m \int_{  \mathbb{R}^3}    U_{\underline{x}_{i} }  U_{\overline{x}_{1} } & =   O_m\left(  e^{-2(1-\delta)  r t}        +      e^{- 2  \pi (1-\delta)(1+\sigma )\sqrt{1-t^2 }    \frac{r}{m}    }  \right ).
\end{align}
\end{lemma}

\begin{proof}
Recalling the definitions $ \overline{x}_{i}, \underline{x}_{i}$, we can know
\begin{align*}
   |  \overline{x}_{i}- \overline{x}_{1}|^2  = 4r^2  ( 1- t^2)    \sin^2 { \frac{ ( i-1 )\pi  }{m}    },
\end{align*}
 \begin{align*}
   |  \overline{x}_{1}- \underline{x}_{1}|^2=4  r^2 t^2,
 \end{align*}
 and
\begin{align*}
 |  \underline{x}_{i}- \overline{x}_{1}|^2  =4  r^2 \Big[  ( 1- t^2)    \sin^2 { \frac{ ( i-1 )\pi  }{m}    }   +t^2     \Big].
\end{align*}
From  the property \eqref{decay} of $U$, we have
\begin{align}  \label{b}
 \int_{  \mathbb{R}^3}  U_{     \overline{x}_{1}  }   U_{\underline{x}_{1}  }    =   O_{m}( e^{  - (1-\delta) |  \overline{x}_{1}- \underline{x}_{1}| } )
    = O_{m}\left( e^{  - 2(1-\delta) r t} \right).
\end{align}

Next,  we calculate
 \begin{align} \label{pro3}
   \sum_{i =2}^m  \int_{  \mathbb{R}^3}  U_{     \overline{x}_{1}  }   U_{\overline{x}_{i}  }  & =  O_{m}\Big(  \sum_{i =2}^m  e^{-(1-\delta)|\overline{x}_{1} -      \overline{x}_{i}| }   \Big).
 \end{align}
Generally, we can suppose that  $m$ is even.
 Then
  \begin{align} \label{pro6}
    \sum_{i =2}^m  e^{-|\overline{x}_1 -      \overline{x}_i| }
&=     \sum_{i =3}^{ m/2}   e^{  - 2r    \sqrt{1-t^2}      \sin{\frac{( i-1 )\pi } {m }            }       }
\nonumber
\\[0.02mm] & \quad
+   \sum_{i = { m/2}  + 1}^{m-1}  e^{  - 2r    \sqrt{1-t^2}      \sin{\frac{( i-1 )\pi } {m }            }       }
+  2  e^{  - 2r    \sqrt{1-t^2}    \sin{\frac{ \pi } {m }            }         }.
 \end{align}
 Consider
 \begin{align*}
 c_{0} \frac{ ( i-1 )\pi  }{m}      \leq   \sin { \frac{ ( i-1 )\pi  }{m}    }     \leq c_{1} \frac{ ( i-1 )\pi  }{m}, \quad \text{for} ~i \in           \big\{ 3, \cdots,    \frac m2 \big\},
 \end{align*}
 with $ \frac 12  <  c_{0} \leq c_{1}\leq 1.$
Then we can derive
\begin{align}\label{pro4}
  \sum_{i=3}^{ m/2}   e^{  - 2r    \sqrt{1-t^2}      \sin{\frac{( i-1 )\pi } {k }            }       }
& \leq    \sum_{i=3}^{ m/2}   e^{  - 2r    \sqrt{1-t^2}     \frac{ c_{0}( i-1 )\pi } {m }                 }      \nonumber
    \\[0.02mm]  & =   \frac{   e^{  - 4r    \sqrt{1-t^2}     \frac{ c_{0}\pi } {m }                     }   -      e^{  - 2r    \sqrt{1-t^2}     \frac{ c_{0}  (\frac m2 ) \pi } {m }                     }          }{       (    1-  e^{  - 2r    \sqrt{1-t^2}     \frac{ c_{0} \pi } {m }                     }   )     }    \nonumber
   \\[0.02mm]
&  = O_{m}\left( e^{  - 2(1+\sigma) \pi  \sqrt{1-t^2}  \frac r m                }\right  ).
\end{align}
With the help of symmetry  of  function $\sin x$,  we have
  $$    \sum_{i = { m/2}  + 1}^{m-1}  e^{  - 2r    \sqrt{1-t^2}      \sin{\frac{( i-1 )\pi } {m }            }       } = O_{m}\left( e^{  - 2(1+\sigma) \pi  \sqrt{1-t^2}  \frac r m                } \right ),  $$ in the same  manner as  \eqref{pro4}.
In the following,  we can get
\begin{align}\label{pro5}
 e^{  - 2r    \sqrt{1-t^2}    \sin{\frac{ \pi } {m }            }         }  &=  e^{  - 2r    \sqrt{1-t^2}}\Big( \frac{ \pi } {m }      + O_{m}( \frac{ \pi^3 } {m^3 } ) \Big) \nonumber
   \\[0.02mm]
&  = e^{  - 2\pi  \sqrt{1-t^2}  \frac r m                }e^{  - 2r    \sqrt{1-t^2}       O_{m}( \frac{ \pi^3 } {m^3 }  )                }\nonumber
   \\[0.02mm]
&  = e^{  - 2\pi  \sqrt{1-t^2}  \frac r m                }   +  O_{m}( e^{  - 2(1+\sigma) \pi  \sqrt{1-t^2}  \frac r m                }  ).
\end{align}
So,
 \begin{align} \label{sum6}
    \sum_{i =2}^m  e^{-(1-\delta) |\overline{x}_{1} -      \overline{x}_{i}| }
&=    O_{m}\left(     e^{  - 2\pi(1-\delta)  \sqrt{1-t^2}  \frac r m                }     +
e^{  - 2(1-\delta)(1+\sigma) \pi  \sqrt{1-t^2}  \frac r m                } \right ).
\end{align}
Following from \eqref{pro3} to \eqref{pro5},  we have
\begin{align}\label{sum0}
   \sum_{i =2}^m  \int_{  \mathbb{R}^3}  U_{     \overline{x}_{1}  }   U_{\overline{x}_{i}  }    =   O_{m}\left(     e^{  - 2\pi(1-\delta)  \sqrt{1-t^2}  \frac r m                }     +
e^{  - 2(1-\delta)(1+\sigma) \pi  \sqrt{1-t^2}  \frac r m                }  \right).
 \end{align}
Finally, we estimate
\begin{align}\label{mathbbI13}
 \sum_{i =1}^m \int_{  \mathbb{R}^3}    U_{\underline{x}_{i} }  U_{\overline{x}_{1} }
&  =     O_{m} \Big(   e^{- (1-\delta)    |  \overline{x}_{1}- \underline{x}_{1}| }   +    \sum_{i =2}^m  e^{- (1-\delta)   |  \underline{x}_{i}- \overline{x}_{1}| } \Big)
\nonumber
   \\[0.02mm]
&  =   O_{m}\Big(   e^{-2(1-\delta) r t} +   \sum_{i =2}^m  e^{-  2 (1-\delta) r \Big[  ( 1- t^2)    \sin^2 { \frac{ ( i-1 )\pi  }{m}    }   +t^2     \Big]^{\frac 12} } \Big)
\nonumber
   \\[0.02mm]
&  =O_{m} \Big(    e^{-2(1-\delta) r t} + 2    e^{-    2 (1-\delta) r \Big[  ( 1- t^2)    \sin^2 { \frac{   \pi  }{m}    }   +t^2     \Big]^{\frac 12} } \Big  ).
\end{align}

Following from  $(r,t)  \in \mathbb{S}_{m}$,  we can get
\begin{align}\label{pro7}
e^{-    2  r \Big[  ( 1- t^2)    \sin^2 { \frac{   \pi  }{m}    }   +t^2     \Big]^{\frac 12} }    & =     e^{-    2  r        ( 1- t^2)^{\frac 12}   \sin { \frac{ \pi  }{m}    } \Big[  1  +   \frac{t^2 }{  ( 1- t^2)    \sin^2 { \frac{  \pi  }{m}    } }        \Big]^{\frac 12} } =  O_{m}( e^{-    2 (1+\sigma )  r        ( 1- t^2)^{\frac 12}     { \frac{ \pi  }{m}    } }).
\end{align}

Then combining \eqref{mathbbI13} and \eqref{pro7} together,   we have
\begin{align}
   \sum_{i =1}^m \int_{  \mathbb{R}^3}    U_{\underline{x}_{i} }  U_{\overline{x}_{1} } & =     O_{m}\Big(  e^{-2(1-\delta)  r t}        +      e^{- 2  \pi (1-\delta)(1+\sigma )\sqrt{1-t^2 }    \frac{r}{m}    } \Big ).
\end{align}
The proof of Lemma \ref{lemma3} is completed.
\end{proof}

\begin{lemma}\label{lemma4}
There exists a  positive constant $C$  such that the following estimate is valid
\begin{align}\label{lemma4.1}
 \int_{ \Omega_{1}^{+}}\Big(\int_{\R^3}\frac{(U_{ \overline{x}_{1}  }+\sum_{i =2}^m U_{ \overline{x}_{i}  } )^2 (y)}{|x-y|}dy\Big)  v_{m} ^2\leq C\Big( e^{- (1-\delta)R  }+\frac{1}{R} \Big)\int_{ \Omega_{1}}v_{m}^2 (x) dx+C \int_{  B_{R}(  \overline{x}_{1}) }v_{m} ^2 (x)d.
\end{align}
\end{lemma}

\begin{proof}
Firstly, letting $ d=\frac{|x-(\overline{x}_{i}-\overline{x}_{1})|}{2}$, if $y\in B_{d}(x)$, we have
\begin{align*}
&\int_{B_{d}(x)}\frac{\sum_{i =2}^m U^2\left(y-(\overline{x}_{i}-\overline{x}_{1})\right) }{|x-y|}dy  \leq C\sum_{i =2}^m \int_{B_{d}(x)}\frac{e^{-2 | y -( \overline{x}_{i}-  \overline{x}_{1} ) | }}{|x-y|}dy
\nonumber
\\[0.02mm]
&\leq C\sum_{i =2}^m \int_{B_{d}(x)}\frac{e^{- | x -( \overline{x}_{i}-  \overline{x}_{1} ) | }}{|x-y|}dy
=C\sum_{i =2}^m e^{- | x -( \overline{x}_{i}-  \overline{x}_{1} ) | }\int_{B_{d}(x)}\frac{1}{|x-y|}dy \nonumber
\\[0.02mm]
&=C\sum_{i =2}^m e^{- | x -( \overline{x}_{i}-  \overline{x}_{1} ) | }| x -( \overline{x}_{i}-  \overline{x}_{1} ) |^2.
\end{align*}
If $y\in B_{d}(\overline{x}_{i}-  \overline{x}_{1})$, we have
\begin{align*}
&\int_{B_{d}(\overline{x}_{i}-  \overline{x}_{1})}\frac{\sum_{i =2}^m U^2\big(y-(\overline{x}_{i}-\overline{x}_{1})\big) }{|x-y|}dy  \leq \sum_{i =2}^m \frac{C}{|x-(\overline{x}_{i}-  \overline{x}_{1})|}\int_{B_{d}(\overline{x}_{i}-  \overline{x}_{1})}e^{-2 | y -( \overline{x}_{i}-  \overline{x}_{1} ) | }dy
\nonumber
\\[0.02mm]
&\leq  \sum_{i =2}^m \frac{C}{|x-(\overline{x}_{i}-  \overline{x}_{1})|}\Big[ e^{- | x -( \overline{x}_{i}-  \overline{x}_{1} ) | } | x -( \overline{x}_{i}-  \overline{x}_{1} )|^2+e^{- | x -( \overline{x}_{i}-  \overline{x}_{1} ) | } | x -( \overline{x}_{i}-  \overline{x}_{1} ) |\nonumber
\\[0.02mm]&\quad+ e^{- | x -( \overline{x}_{i}-  \overline{x}_{1} ) | } +1  \big]\nonumber
\\[0.02mm]
&\leq C\sum_{i =2}^m \Big[ e^{- | x -( \overline{x}_{i}-  \overline{x}_{1} ) | } \sum_{l=-1 }^1| x -( \overline{x}_{i}-  \overline{x}_{1} ) |^l+ \frac{1}{|x-(\overline{x}_{i}-  \overline{x}_{1})|}\Big].
\end{align*}
If $y\in \mathbb{R}^3 \setminus B_{d}(x) \cup B_{d}( \overline{x}_{i}- \overline{x}_{1} ) $, we have $|y-x|\geq\frac{|x-(\overline{x}_{i}-  \overline{x}_{1})|}{2}$ and $|y-(\overline{x}_{i}-  \overline{x}_{1})|\geq\frac{|x-(\overline{x}_{i}-  \overline{x}_{1})|}{2}$.

If $|y-(\overline{x}_{i}-  \overline{x}_{1})|\geq2|x-(\overline{x}_{i}-  \overline{x}_{1})|$, then $|x-y|\geq\frac{1}{2}|x-(\overline{x}_{i}-  \overline{x}_{1})|.$

If $|y-(\overline{x}_{i}-  \overline{x}_{1})|\leq 2|x-(\overline{x}_{i}-  \overline{x}_{1})|$, then $|y-x|\geq\frac{1}{4}|x-(\overline{x}_{i}-  \overline{x}_{1})|.$

In summary $|y-(\overline{x}_{i}-  \overline{x}_{1})|\geq\frac{1}{4}|x-(\overline{x}_{i}-  \overline{x}_{1})|$.

Therefore,
\begin{align*}
&\int_{\mathbb{R}^3 \setminus B_{d}(x) \cup B_{d}( \overline{x}_{i}- \overline{x}_{1} )}\frac{\sum_{i =2}^m U^2\big(y-(\overline{x}_{i}-\overline{x}_{1})\big) }{|x-y|}dy  \leq C \sum_{i =2}^m \int_{\mathbb{R}^3 \setminus B_{d}(x) \cup B_{d}( \overline{x}_{i}- \overline{x}_{1} )}\frac{e^{-2 | y -( \overline{x}_{i}-  \overline{x}_{1} ) | }}{|x-(\overline{x}_{i}-  \overline{x}_{1})|}dy \nonumber
\\[0.02mm]
&\leq \sum_{i =2}^m \frac{C}{|x-(\overline{x}_{i}-  \overline{x}_{1})|}\int_{\mathbb{R}^3 \setminus  B_{d}( \overline{x}_{i}- \overline{x}_{1} )}e^{-2 | y -( \overline{x}_{i}-  \overline{x}_{1} ) | }dy
\nonumber
\\[0.02mm]
&\leq  \sum_{i =2}^m \frac{C}{|x-(\overline{x}_{i}-  \overline{x}_{1})|}\Big[ e^{- | x -( \overline{x}_{i}-  \overline{x}_{1} ) | } | x -( \overline{x}_{i}-  \overline{x}_{1} ) |^2+e^{- | x -( \overline{x}_{i}-  \overline{x}_{1} ) | } | x -( \overline{x}_{i}-  \overline{x}_{1} ) |+ e^{- | x -( \overline{x}_{i}-  \overline{x}_{1} ) | }  \Big]\nonumber
\\[0.02mm]
&\leq C\sum_{i =2}^m \Big[ e^{- | x -( \overline{x}_{i}-  \overline{x}_{1} ) | } \sum_{l=-1 }^1| x -( \overline{x}_{i}-  \overline{x}_{1} ) |^l\Big].
\end{align*}
Hence we have
\begin{align*}
&\int_{\mathbb{R}^3}\frac{\sum_{i =2}^m U^2\big(y-(\overline{x}_{i}-\overline{x}_{1})\big) }{|x-y|}dy  \leq C\sum_{i =2}^m \Big[ e^{- | x -( \overline{x}_{i}-  \overline{x}_{1} ) | } \sum_{l=-1 }^2| x -( \overline{x}_{i}-  \overline{x}_{1} ) |^l+ \frac{1}{|x-(\overline{x}_{i}-  \overline{x}_{1})|}\Big].
\end{align*}
Therefore we get
\begin{align*}
&\int_{ \Omega_{1}^{+}}\Big(\int_{\R^3}\frac{(U_{ \overline{x}_{1}  }+\sum_{i =2}^m U_{ \overline{x}_{i}  } )^2 (y)}{|x-y|}dy\Big)  v_{m} ^2=\int_{ \Omega_{1}^{+}}\Big(\int_{\R^3}\frac{\Big[U(y)+\sum_{i =2}^m U\big( y-(\overline{x}_{i}-\overline{x}_{1}) \big) \Big]^2 (y)}{|x-y-\overline{x}_{1}|}dy\Big)  v_{m} ^2
\nonumber
\\[0.02mm]
&=\int_{ \Omega_{1}^{+}-\overline{x}_{1}}\Big(\int_{\R^3}\frac{\Big[U(y)+\sum_{i =2}^m U\big ( y-(\overline{x}_{i}-\overline{x}_{1}) \big) \Big]^2 (y)}{|x-y|}dy\Big)  v_{m} ^2(x+\overline{x}_{1})
\nonumber
\\[0.02mm]
& \leq C\Big(  \int_{ \Omega_{1}^{+}-\overline{x}_{1}}\Big(\int_{\R^3}\frac{U^2 (y)}{|x-y|}dy\Big)  v_{m} ^2(x+\overline{x}_{1}) +\int_{ \Omega_{1}^{+}-\overline{x}_{1}}\Big(\int_{\R^3}\frac{\Big[\sum_{i =2}^m U\big ( y-(\overline{x}_{i}-\overline{x}_{1}) \big) \Big]^2 (y)}{|x-y|}dy\Big)  v_{m} ^2(x+\overline{x}_{1})\Big)
\nonumber
\\[0.02mm]
& \leq C\int_{ \Omega_{1}^{+}-\overline{x}_{1}}\Big(\Big[ e^{- | x| } \sum_{l=-1 }^2| x |^l+ \frac{1}{|x|}\Big]+\sum_{i =2}^m \Big[ e^{- | x -( \overline{x}_{i}-  \overline{x}_{1} ) | } \sum_{l=-1 }^2| x -( \overline{x}_{i}-  \overline{x}_{1} ) |^l+ \frac{1}{|x-(\overline{x}_{i}-  \overline{x}_{1})|}\Big]\Big)v_{m} ^2(x+\overline{x}_{1})\nonumber
\\[0.02mm]
&=C\int_{ \Omega_{1}^{+}}\Big(\Big[ e^{- | x-\overline{x}_{1}| } \sum_{l=-1 }^2| x-\overline{x}_{1} |^l+ \frac{1}{|x-\overline{x}_{1}|}\Big]+\sum_{i =2}^m \Big[ e^{- | x - \overline{x}_{i} | } \sum_{l=-1 }^2| x - \overline{x}_{i}|^l+ \frac{1}{|x-\overline{x}_{i}|}\big]\Big)v_{m} ^2(x)\nonumber
\\[0.02mm]
&\leq C\int_{ \Omega_{1}^{+}}\sum_{i =1}^m \Big[ e^{- | x - \overline{x}_{i} | } \sum_{l=-1 }^2| x - \overline{x}_{i}|^l+ \frac{1}{|x-\overline{x}_{i}|}\Big]v_{m} ^2(x)\nonumber
\\[0.02mm]
&=C \int_{ \Omega_{1}^{+}\cap B_{R}(  \overline{x}_{1}) }\sum_{i =1}^m \Big[ e^{- | x - \overline{x}_{i} | } \sum_{l=-1 }^2| x - \overline{x}_{i}|^l+ \frac{1}{|x-\overline{x}_{i}|}\Big]v_{m} ^2(x)\nonumber
\\[0.02mm]
&\quad+C\int_{ \Omega_{1}^{+}\setminus B_{R}(  \overline{x}_{1})}\sum_{i =1}^m \Big[ e^{- | x - \overline{x}_{i} | } \sum_{l=-1 }^2| x - \overline{x}_{i}|^l+ \frac{1}{|x-\overline{x}_{i}|}\Big]v_{m} ^2(x)\nonumber
\\[0.02mm]
&\leq C\big( e^{- (1-\delta)R  }+\frac{1}{R} \big)\int_{ \Omega_{1}}v_{m}^2 (x) dx+C \int_{  B_{R}(  \overline{x}_{1}) }v_{m} ^2 (x)dx.
\end{align*}
Indeed,
if $i=1$ and $ x\in\Omega_{1}^{+}\cap B_{R}(  \overline{x}_{1}) $, we can get
\begin{align*}
&\sum_{i =1}^m \Big[ e^{- | x - \overline{x}_{i} | } \sum_{l=-1 }^2| x - \overline{x}_{i}|^l+ \frac{1}{|x-\overline{x}_{i}|}\Big]=\Big[ e^{- | x - \overline{x}_{1} | } \sum_{l=-1 }^2| x - \overline{x}_{1}|^l+ \frac{1}{|x-\overline{x}_{1}|}\Big]=o(1);
\end{align*}
if $i\neq 1$ and $ x\in\Omega_{1}^{+}\cap B_{R}(  \overline{x}_{1}) $, we can get
\begin{align*}
&\sum_{i =2}^m \Big[ e^{- | x - \overline{x}_{i} | } \sum_{l=-1 }^2| x - \overline{x}_{i}|^l+ \frac{1}{|x-\overline{x}_{i}|}\Big]
\nonumber
\\[0.02mm]
&\leq C\sum_{i =2}^m \Big[ e^{- r\sqrt{1-t^{2}}\sin\frac{|i-1|\pi}{m} } \sum_{l=-1 }^2| r\sqrt{1-t^{2}}\sin\frac{|i-1|\pi}{m}|^l+ \frac{1}{|r\sqrt{1-t^{2}}\sin\frac{|i-1|\pi}{m}|}\Big]\nonumber
\\[0.02mm]
&\leq C\sum_{i =2}^m \Big[ e^{-\frac{1}{2} r\sqrt{1-t^{2}}\frac{|i-1|\pi}{m} } \sum_{l=-1 }^2| r\sqrt{1-t^{2}}\frac{|i-1|\pi}{m}|^l+ \frac{1}{|r\sqrt{1-t^{2}}\frac{|i-1|\pi}{m}|}\Big]\nonumber
\\[0.02mm]
&=o(1).
\end{align*}
So, we have
\begin{align*}
&\int_{ \Omega_{1}^{+}\cap B_{R}(  \overline{x}_{1}) }\sum_{i =1}^m \Big[ e^{- | x - \overline{x}_{i} | } \sum_{l=-1 }^2| x - \overline{x}_{i}|^l+ \frac{1}{|x-\overline{x}_{i}|}\Big]v_{m} ^2(x)\leq C \int_{  B_{R}(  \overline{x}_{1}) }v_{m} ^2 (x)dx.
\end{align*}
Similarly,
if $i=1$ and $x\in{ \Omega_{1}^{+}\setminus B_{R}(  \overline{x}_{1})}$, we can get
\begin{align*}
&\sum_{i =1}^m \Big[ e^{- | x - \overline{x}_{i} | } \sum_{l=-1 }^2\Big| x - \overline{x}_{i}\Big|^l+ \frac{1}{|x-\overline{x}_{i}|}\Big]\nonumber
\\[0.02mm]
&= \Big[ e^{- | x - \overline{x}_{1} | } \sum_{l=-1 }^2| x - \overline{x}_{1}|^l+ \frac{1}{|x-\overline{x}_{1}|}\Big]
\leq C\Big( e^{- (1-\delta)R  }+\frac{1}{R}  \Big),
\end{align*}
if $i\neq 1$ and $x\in{ \Omega_{1}^{+}\setminus B_{R}(  \overline{x}_{1})}$, we can get
\begin{align*}
&\sum_{i =2}^m \Big[ e^{- | x - \overline{x}_{i} | } \sum_{l=-1 }^2| x - \overline{x}_{i}|^l+ \frac{1}{|x-\overline{x}_{i}|}\Big]\nonumber
\\[0.02mm]
&= \sum_{i =2}^m \Big[ e^{- | x - \overline{x}_{i} | } \sum_{l=-1 }^2| x - \overline{x}_{i}|^l+ \frac{1}{|x-\overline{x}_{i}|}\Big]
\nonumber
\\[0.02mm]
&\leq C\sum_{i =2}^m \Big[ e^{- | x - \overline{x}_{i} | }\Big(  \frac{1}{|\overline{x}_{i}-\overline{x}_{1}|}+1+| x - \overline{x}_{i} |+ | x - \overline{x}_{i} |^2\Big) + \frac{1}{|\overline{x}_{i}-\overline{x}_{1}|}\Big]
\nonumber
\\[0.02mm]
&\leq C\sum_{i =2}^m e^{-\frac{|\overline{x}_{i}-\overline{x}_{1}|}{4}}\Big[ e^{- \frac{| x - \overline{x}_{i} | }{4}}| x - \overline{x}_{i} |^2\Big]+ C\frac{m\ln m}{\pi r \sqrt{1-t^{2}}}\nonumber
\\[0.02mm]
&\leq C e^{-\frac{r\pi\sqrt{1-t^{2}}}{4m}} + C \frac{1}{R}\leq C\Big( e^{- (1-\delta)R  } \Big)+ C\frac{1}{R}.
\end{align*}
So, we have
\begin{align*}
&\int_{ \Omega_{1}^{+}\setminus B_{R}(  \overline{x}_{1})}\sum_{i =1}^m \Big[ e^{- | x - \overline{x}_{i} | } \sum_{l=-1 }^2| x - \overline{x}_{i}|^l+ \frac{1}{|x-\overline{x}_{i}|}\Big]v_{m} ^2(x)\leq C\Big( e^{- (1-\delta)R  }+\frac{1}{R} \Big)\int_{ \Omega_{1}}v_{m}^2 (x) dx.
\end{align*}
\end{proof}

\vspace{3mm}
\section{Energy expansion }\label{sb}

In this section, we will estimate $I(W_{r,t} ).$
\begin{proposition}\label{func}
 For all $(r,t) \in \mathbb{S}_{m} $, there exists a  small constant $\tau>0$  such that
\begin{align}\label{2.1}
  I(W_{r,t} )
   &=   m\Big(     \frac{A_{2}}{16\pi}+ \frac{A_{1}}{r^q}  -  \frac{A_{1}^2}{16\pi^2\sqrt{1-t^{2}}}\frac{m\ln m}{r}-\frac{m}{16\pi^2}\frac{1 }{r\sqrt{1-t^{2}}}\ln\frac{\pi}{t}A_{1}^{2}   \Big)  \nonumber
   \\[0.02mm]
  &\quad+m O_{m}(\frac{1}{r^{q+\tau}})+m O_{m}\Big(\frac{m^{2}}{ r^{2}(1-t^{2})}\Big)+\frac{m^{2} }{r\sqrt{1-t^{2}}}\ln\frac{\pi}{t}O_{m} \big(t^{2} |\ln t|^{-1}\big),
\end{align}
where
\begin{align}\label{A1A2}
A_{1} =  \int_{  \mathbb{R}^3}  U^2,   \quad
A_{2}=\int_{\R^3}\int_{\R^3}\frac{U^2(y)U^2(x)}{|x-y|}.
\end{align}
\end{proposition}

\medskip

\begin{proof}
By direct computations, we have
\begin{align}\label{energyexpan}
 I(W_{r,t} ) &= \frac 12 \int_{  \mathbb{R}^3}   \Big\{  |  \nabla  W_{r,t} |^2 + V(|y|) W_{r,t}^2  \Big\} - \frac{1}{32\pi}\int_{\R^3}\int_{\R^3}\frac{(W_{r,t})^2(y)(W_{r,t})^2(x)}{|x-y|} \nonumber
 \\[0.02mm]
&  = \frac 12 \int_{  \mathbb{R}^3}   \Big\{ |  \nabla  W_{r,t} |^2  +   | W_{r,t}| ^2 \Big\}  \nonumber
+ \frac 12 \int_{  \mathbb{R}^3}  \big( V(|y|)- 1 \big )  | W_{r,t}| ^2\\[0.02mm]
   &\quad  - \frac{1}{32\pi}\int_{\R^3}\int_{\R^3}\frac{(W_{r,t})^2(y)(W_{r,t})^2(x)}{|x-y|}   \nonumber
   \\[0.02mm]
  &  = \mathbb{I}_{1} +  \mathbb{I}_{2} - \mathbb{I}_{3}.
\end{align}
\par
With the help of the symmetry and the equation for $U$, we can derive
 \begin{align}\label{mathbbI1}
 \mathbb{I}_{1} &= \frac 12 \int_{  \mathbb{R}^3}   \big( \{ |  \nabla  W_{r,t} |^2  +   | W_{r,t}| ^2 \big)    = \frac 12 \int_{  \mathbb{R}^3}    \big (   - \Delta W_{r,t} + W_{r,t}  \big ) W_{r,t} \nonumber
 \\[0.02mm]
&  = \frac 12 \int_{  \mathbb{R}^3} \sum_{j=1}^m  \big (    - \Delta  U_{\overline{x}_{j}  } + U_{\overline{x}_{j}  }       - \Delta U_{\underline{x}_{j}  }  + U_{\underline{x}_{j}  }   \big) \cdot  \sum_{i =1}^m  \big ( U_{\overline{x}_{i} } + U_{\underline{x}_{i} }    \big) \nonumber
 \\[0.02mm]
&  = \frac 12   \sum_{j=1}^m  \sum_{ i =1}^m   \int_{  \mathbb{R}^3}     \big (   \Psi_{ U_{\overline{x}_{j}}}U_{\overline{x}_{j}}+\Psi_{ U_{\underline{x}_{j}  }}
   U_{\underline{x}_{j}  }     \big) \cdot    \big ( U_{\overline{x}_{i} } + U_{\underline{x}_{i} }    \big)  \nonumber
   \\[0.02mm]
&  = \frac 12   \sum_{j=1}^m  \sum_{i =1}^m   \int_{  \mathbb{R}^3}     \big (    \Psi_{ U_{\overline{x}_{j}}}U_{\overline{x}_{j}}   U_{\overline{x}_{i}  }
  +   \Psi_{ U_{\underline{x}_{j}  }}
   U_{\underline{x}_{j}  }   U_{\overline{x}_{i}  }  + \Psi_{ U_{\overline{x}_{j}}}U_{\overline{x}_{j}}  U_ { \underline{x}_{i}  } + \Psi_{ U_{\underline{x}_{j}  }}
   U_{\underline{x}_{j}  }   U_{\underline{x}_{i}  }      \big)  \nonumber
   \\[0.02mm]
&  =   \frac{m}{8\pi}\int_{\R^3}\int_{\R^3}\frac{U^2(y)U^2(x)}{|x-y|}   + \frac{m}{8\pi} \sum_{i =2}^m  \int_{  \mathbb{R}^3} \frac{U^2_{\overline{x}_{1}  }}{|x-y|} U_{\overline{x}_{1}  }   U_{\overline{x}_{i}  }  +   \frac{m}{8\pi} \sum_{i =1}^m  \int_{  \mathbb{R}^3} \frac{U^2_{\overline{x}_{1}  }}{|x-y|} U_{\overline{x}_{1}  }    U_{\underline{x}_{i} }  .
 \end{align}

An elementary calculation  is as follows :
\begin{equation}\label{t}
 \frac{1}{|x-\overline{x}_{1}|^{t}}=\frac{1}{|\overline{x}_{1}|^{t}}\Big(1+O\Big(\frac{|x|}{|\overline{x}_{1}|}\Big)\Big), \ x\in B_{\frac{|\overline{x}_{1}|}{2}}(0),\forall\ t>0,
\end{equation}
 With symmetry and  Lemma \ref{lemma1},  we can calculate $\mathbb{I}_{2}$:
\begin{align}\label{mathbbI2}
\mathbb{I}_{2} & = \frac 12 \int_{  \mathbb{R}^3}  \big( V(|y|)- 1 \big )  | W_{r,t}| ^2
 \nonumber
   \\[0.02mm]
&  =      m   \int_{  \Omega_{1}^+}    \big( V(|y|)- 1 \big )  \Big( U_{ \overline{x}_{1}}    + U_{ \underline{x}_{1}}   + \sum_{j =2}^m  U_{ \underline{x}_{j}} +\sum_{j =2}^m  U_{ \overline{x}_{j}}  \Big)^2 \nonumber
   \\[0.02mm]
&  =     m   \int_{  \Omega_{1}^+}    \big( V(|y|)- 1 \big )   \Big( U_{ \overline{x}_{1}} + O_{m}\big(  e^{ -\frac 12 t r}   e^{-\frac 12 | y -   \overline{x}_{1}  | }    \nonumber
   \\[0.02mm]  &\quad
+   e^{- \frac 12 \sqrt{1-t^2} r\frac{\pi}{m}}e^{   -\frac 12 |y-\overline{x}_{1}| }+       e^{- \frac 12\sqrt{1-t^2} r\frac{\pi}{m}}  e^{   -\frac 12 |y-\overline{x}_{1}| }   \big) \Big)^2\nonumber
   \\[0.02mm]
&  =    m   \int_{  \Omega_{1}^+}  \big( V(|y|)- 1 \big )   U^2_{ \overline{x}_{1}}
 + m   O_{m}\Big(   \int_{  \Omega_{1}^+}  \big( V(|y|)- 1 \big )            e^{   -\frac 12 |y-\overline{x}_{1}| }      U_{ \overline{x}_{1}} \Big)  \nonumber
   \\[0.02mm]
&  =  m \Big(  \frac{A_{1}}{r^q}+ O_{m}(\frac{1}{r^{q+\tau}})\Big),
\end{align}
where $A_{1}   $ is defined in \eqref{A1A2}  and the last  equality is valid because of the asymptotic expression of $V(y)$ and \eqref{t}.
Moreover, for  any $ (r,t)\in S_{m}, $ and $ y\in \Omega_{1}^+$  we can derive
\begin{align}\label{t1}
U_{\overline{x}_{1}  }    U_{\underline{x}_{1} } \leq &C e^{- | \overline{x}_{1}  -  \underline{x}_{1}|      }   = C e^{-2tr },
\end{align}

\begin{align}\label{t2}
  U_{\overline{x}_{1} }  \sum_{j=2}^m U_{\overline{x}_{j}  }   \leq C \sum_{j=2}^m e^{- | \overline{x}_{1}  -  \overline{x}_{j}|      }   \leq C e^{-2\pi  \sqrt{1-t^2}  \frac r m }+ Ce^{-2(1+\tau) \pi  \sqrt{1-t^2}  \frac r m },
\end{align}
and
\begin{align}\label{t3}
U_{\overline{x}_{1}  }     \sum_{j=2}^m U_{\underline{x}_{j} } \leq C\sum_{j=2}^m  e^{- | \overline{x}_{1}  -  \underline{x}_{j}|      }   \leq C e^{-2\pi  \sqrt{1-t^2}  \frac r m }.
\end{align}
By \eqref{t1}, \eqref{t2}, \eqref{t3}, and Lemma \ref{lemma1}, we have
\begin{align*}
 \mathbb{I}_{3} & = \frac{1}{32\pi}\int_{\R^3}\int_{\R^3}\frac{(W_{r,t})^2(y)(W_{r,t})^2(x)}{|x-y|} \nonumber
   \\[0.02mm]
&  = \frac{1}{4}\int_{\R^3}\Psi_{ W_{r,t}}(W_{r,t})^2 dx
 \nonumber
   \\[0.02mm]
&  =    \frac{m}{2}\int_{\Omega_1^+}\Psi_{ W_{r,t}}\Big( U_{\overline{x}_{1} }+U_{\underline{x}_{1} }+ \sum_{j=2}^m U_{\overline{x}_{j} }+ \sum_{i=2}^m U_{\underline{x}_{i}}\Big)^2 dx\nonumber
   \\[0.02mm]
&  =    \frac{m}{2}\int_{\Omega_1^+}\Psi_{ W_{r,t}}\Big( U_{\overline{x}_{1} }^2+U_{\underline{x}_{1} }^2+ (\sum_{j=2}^m U_{\overline{x}_{j} })^2+ (\sum_{i=2}^m U_{\underline{x}_{i}})^2 +2U_{\overline{x}_{1} } U_{\underline{x}_{1} }+2\sum_{j=2}^m\sum_{i=2}^m U_{\overline{x}_{j} }U_{\underline{x}_{i}}\Big) dx\nonumber
   \\[0.02mm]
& \quad+ \frac{m}{2}\int_{\Omega_1^+}\Psi_{ W_{r,t}}\Big( 2U_{\overline{x}_{1} }\sum_{j=2}^m U_{\overline{x}_{j} }+2U_{\overline{x}_{1} }\sum_{i=2}^m U_{\underline{x}_{i}}+2U_{\underline{x}_{1}}\sum_{j=2}^m U_{\overline{x}_{j} }+2U_{\underline{x}_{1}}\sum_{i=2}^m U_{\underline{x}_{i}}\Big) dx
\nonumber
   \\[0.02mm]
& = \frac{m}{2}\int_{\Omega_{1}^+}\Psi_{ W_{r,t}}U_{\overline{x}_{1}}^2+m O_{m}\Big(  e^{ -   \tau  t r} + e^{- \tau \sqrt{1-t^2} r\frac{\pi}{m}}\Big).
 \end{align*}
\par
Moreover, we have
 \begin{align*}
 \frac{m}{2}\int_{\Omega_{1}^+}\Psi_{ W_{r,t}}U_{\overline{x}_{1}}^2
 &= \frac{m}{2}\int_{\R^3}\Psi_{ W_{r,t}}U_{\overline{x}_{1}}^2-\frac{m}{2}\int_{  \mathbb{R}^3 \setminus  \Omega_{1}^+     } \Psi_{ W_{r,t}}U_{\overline{x}_{1}}^2
 \nonumber
   \\[1mm]
&\leq\frac{m}{2}\int_{\R^3}\Psi_{ W_{r,t}}U_{\overline{x}_{1}}^2-\frac{m}{2}\int_{  \mathbb{R}^3 \setminus  B_\frac{|\overline{x}_{2}-\overline{x}_{1}|}{2} (\overline{x}_{1})    } \Psi_{ W_{r,t}}U_{\overline{x}_{1}}^2\nonumber
   \\[1mm]
&\leq\frac{m}{2}\int_{\R^3}\Psi_{ W_{r,t}}U_{\overline{x}_{1}}^2+O_{m}\Big(   e^{- \tau \sqrt{1-t^2} r\frac{\pi}{m}}\Big).
 \end{align*}
We calculate the first term as follows:
\begin{align*}
\frac{m}{2}\int_{\R^3}\Psi_{ W_{r,t}}U_{\overline{x}_{1}}^2
&=\frac{m}{2}\int_{\R^3}\Psi_{ U_{\overline{x}_{1}}}W_{r,t}^2
\nonumber
 \\[1mm]
&=\frac{m}{2}\int_{\R^3}\Psi_{ U^2}U^2+m\int_{\R^3}\Psi_{ U_{\overline{x}_{1}}}U_{\overline{x}_{1}}\sum_{i=2}^m U_{\overline{x}_{i}}+m\int_{\R^3}\Psi_{ U_{\overline{x}_{1}}}U_{\overline{x}_{1}}\sum_{i=1}^m U_{\underline{x}_{i}}
\nonumber
\\[1mm]
&\quad+\frac{m}{2}\int_{\R^3}\Psi_{ U_{\overline{x}_1}}\Big(\sum_{i=2}^m U_{\overline{x}_{i}} \Big)^2+\frac{m}{2}\int_{\R^3}\Psi_{ U_{\overline{x}_{1}}}\Big(\sum_{i=1}^m U_{\underline{x}_{i}} \Big)^2\\
&\quad+m \sum_{j=2}^m \sum_{i=1}^m \int_{\R^3}\Psi_{ U_{\overline{x}_{1}}}U_{\overline{x}_{j}}U_{\underline{x}_{i}}.
\end{align*}
From Lemma \ref{lemma1}, Lemma \ref{lemma2} and \eqref{sum6}, we have
 \begin{align*}
\frac{m}{2}\int_{\R^3}\Psi_{ U_{\overline{x}_{1}}}\Big(\sum_{i=2}^m U_{\overline{x}_{i}} \Big)^2 &=\frac{m}{2}\sum_{i=2}^m\int_{\R^3}\Psi_{ U_{\overline{x}_{1}}}\Big( U_{\overline{x}_{i}} \Big)^2+\frac{m}{2}\sum_{i,j=2 \atop i\neq j }^m\int_{\R^3}\Psi_{ U_{\overline{x}_{1}}}U_{\overline{x}_{i}} U_{\overline{x}_{j}}\nonumber
   \\[0.02mm]
&=\frac{m A_{1}^2}{16\pi^2}\frac{m\ln m}{r\sqrt{1-t^{2}}}+ m O_{m}\Big(\frac{m^{2}}{r^{2}(1-t^{2})} \Big)+m O_{m} \Big(  \sum_{i,j=2 \atop i\neq j }^m e^{-(1- \tau)|\overline{x}_{i}-\overline{x}_{j} | }\Big)\nonumber
   \\[0.02mm]
&=\frac{m A_{1}^2}{16\pi^2}\frac{m\ln m}{r\sqrt{1-t^{2}}}+ m O_{m}\Big(\frac{m^{2}}{r^{2}(1-t^{2})} \Big)+m O_{m} \Big(m  \sum_{i=2 }^{m}e^{-(1- \tau)|\overline{x}_{i}-\overline{x}_{1} | }\Big)
\nonumber
   \\[0.02mm]
&=\frac{m A_{1}^2}{16\pi^2}\frac{m\ln m}{r\sqrt{1-t^{2}}}+ m O_{m}\Big(\frac{m^{2}}{r^{2}(1-t^{2})} \Big)+m O_{m} \Big(   e^{- \tau \sqrt{1-t^2} r\frac{\pi}{m}}\Big).
\end{align*}
Also, from the results of ~Medina~and~Musso~in \cite{MaMo}, we know
\begin{align*}
\sum_{i=1}^m\frac{1}{|\overline{x}_{1}-\underline{x}_{i}|}
&=\frac{1}{r t}\Big(\sum_{i=1}^\frac{m-1}{2}\frac{1}{\big(\frac{(1-t^{2})}{t^{2}}\frac{(i-1)^{2}\pi^{2}}{m^{2}}+1 \big)^{\frac{1}{2}}} \Big)+O\Big( |\ln t|^{-1} \Big)
\nonumber
   \\[0.02mm]
&=\frac{m }{r\pi\sqrt{1-t^{2}}}\ln\frac{\pi}{t}\Big(1+O\big(t^{2} |\ln t|^{-1}\big) \Big),
\end{align*}
so we have
\begin{align*}
\frac{m}{2}\int_{\R^3}\Psi_{ U_{\overline{x}_{1}}}\big(\sum_{i=1}^m U_{\underline{x}_{i}} \big)^2
&=\frac{m}{2}\sum_{i=1}^m\int_{\R^3}\Psi_{ U_{\overline{x}_{1}}}\big( U_{\underline{x}_{i}} \big)^2+\frac{m}{2}\sum_{i,j=1\atop i\neq j}^m\int_{\R^3}\Psi_{ U_{\overline{x}_1}}U_{\underline{x}_{i}}U_{\underline{x}_{j}}\nonumber
   \\[0.02mm]
&= \frac{m}{2}\sum_{i=1}^m\int_{\R^3}\Psi_{ U_{\overline{x}_{1}}}\big( U_{\underline{x}_{i}} \big)^2+ m O\Big( \sum_{i,j=1 \atop i\neq j }^m e^{-(1- \tau)|\underline{x}_{i}-\underline{x}_{j} | } \Big)\nonumber
   \\[0.02mm]
&= \frac{m}{2}\sum_{i=1}^m\int_{\R^3}\Psi_{ U_{\overline{x}_{1}}}\big( U_{\underline{x}_{i}} \big)^2+ m O\Big( m\sum_{i=1 }^m e^{-(1- \tau)|\underline{x}_{i}-\underline{x}_{1} | }  \Big) \nonumber
   \\[0.02mm]
&= \frac{m}{16\pi}\sum_{i=1}^m\frac{1}{|\overline{x}_{1}-\underline{x}_{i}|}\int_{\R^3}U^2(x)\int_{\R^3}U^2(y)
     \nonumber
   \\[0.02mm]
&\quad+ m O\left(\sum_{i=1}^m\frac{1}{|\overline{x}_{1}-\underline{x}_{i}|^2}\right)+ m O\big( m\sum_{i=1 }^m e^{-(1- \tau)|\underline{x}_{i}-\underline{x}_{1} | }  \big)\nonumber
   \\[0.02mm]
&= \frac{m}{16\pi}\Big[\frac{m }{r\pi\sqrt{1-t^{2}}}\ln\frac{\pi}{t}\Big(1+O_{m}\big(t^{2} |\ln t|^{-1}\big) \Big)\Big]      \int_{\R^3}U^2(x)\int_{\R^3}U^2(y)\nonumber
   \\[0.02mm]
&\quad+ m O\Big(\sum_{i=2}^m\frac{1}{|\overline{x}_{1}-\overline{x}_{i}|^2}\Big)+m O_{m} \Big(   e^{- \tau \sqrt{1-t^2} r\frac{\pi}{m}}\Big)
\nonumber
   \\[0.02mm]
&=\frac{m}{16\pi}\frac{m }{r\pi\sqrt{1-t^{2}}}\ln\frac{\pi}{t}A_{1}^{2} +\frac{m^{2} }{r\pi\sqrt{1-t^{2}}}\ln\frac{\pi}{t}O_{m} \big(t^{2} |\ln t|^{-1}   \big)\nonumber
   \\[0.02mm]
&\quad+ m O_{m}\Big(\frac{m^{2}}{r^{2}(1-t^{2})}\Big)+m O_{m} \Big(   e^{- \tau \sqrt{1-t^2} r\frac{\pi}{m}}\Big),
\end{align*}
and
 \begin{align*}
 m \sum_{j=2}^m \sum_{i=1}^m \int_{\R^3}\Psi_{ U_{\overline{x}_{1}}}U_{\overline{x}_{j}}U_{\underline{x}_{i}}=m \sum_{j=2}^m \sum_{i=1}^m O_{m}\left(e^{-\left(1-\delta\right) | \overline{x}_{i}  -  \underline{x}_{j}|      } \right)=m O_{m} \Big(   e^{- \tau \sqrt{1-t^2} r\frac{\pi}{m}}\Big).
 \end{align*}
Hence we have
 \begin{align}  \label{mathbbI3}
 \mathbb{I}_{3}&=\frac{m}{2}\int_{\R^3}\Psi_{ U^2}U^2+m\int_{\R^3}\Psi_{ U_{\overline{x}_{1}}}U_{\overline{x}_{1}}\sum_{i=2}^m U_{\overline{x}_{i}}+m\int_{\R^3}\Psi_{ U_{\overline{x}_{1}}}U_{\overline{x}_{1}}\sum_{i=1}^m U_{\underline{x}_{i}}
\nonumber
   \\[1mm]
&\quad+\frac{m A_{1}^2}{16\pi^2}\frac{m\ln m}{r\sqrt{1-t^{2}}}+ \frac{m}{16\pi}\frac{m }{r\pi\sqrt{1-t^{2}}}\ln\frac{\pi}{t}A_{1}^{2} +\frac{m^{2} }{r\pi\sqrt{1-t^{2}}}\ln\frac{\pi}{t}O_{m} \big(t^{2} |\ln t|^{-1}   \big)
\nonumber
   \\[1mm]
&\quad+m O_{m}\Big(\frac{m^{2}}{r^{2}(1-t^{2})}+ e^{- \tau \sqrt{1-t^2} r\frac{\pi}{m}}\Big).
\end{align}
Combining \eqref{energyexpan},  \eqref{mathbbI1}, \eqref{mathbbI2}, \eqref{mathbbI3},  we have
\begin{align*} 
 I(W_{r,t} )
 & = \mathbb{I}_{1} +  \mathbb{I}_{2} - \mathbb{I}_{3}  \nonumber
   \\[1mm]
   & =\frac{m A_{2}}{16\pi}+m \frac{A_{1}}{r^q}-\frac{m A_{1}^2}{16\pi^2}\frac{m\ln m}{r\sqrt{1-t^{2}}}
     -\frac{m^{2}}{16\pi^2}\frac{1 }{r\sqrt{1-t^{2}}}\ln\frac{\pi}{t}A_{1}^{2} \nonumber
   \\[1mm]
    &\quad+O_{m}(\frac{1}{r^{q+\tau}})+m O_{m}\Big(\frac{m^{2}}{ r^{2}(1-t^{2})}\Big)-\frac{m^{2} }{r\sqrt{1-t^{2}}}\ln\frac{\pi}{t}O_{m} \big(t^{2} |\ln t|^{-1}   \big),
\end{align*}
where    $A_{1}, A_{2}$  are  defined in \eqref{A1A2}.
We complete the proof of Proposition \ref{func}.
 \end{proof}
\def \newblockProc{}
 \vspace{3mm}

Acknowledgment. The authors would like to thank Chunhua Wang from Central China Normal University for the
helpful discussion with her. This paper was supported by NSFC grants (No.12071169) and the Fundamental Research Funds for the Central Universities(No.KJ02072020-
0319).

\end{document}